\documentclass[12pt,a4paper,reqno]{amsart} 
\usepackage{amssymb}
\usepackage{latexsym}
\usepackage{amsmath}
\usepackage{mathrsfs}
\usepackage{nccmath}    
\usepackage[T1]{fontenc} 
\usepackage{cite}
\usepackage{color}

\usepackage{calc}                   

\newcommand{\scal}[2]{\langle #1,#2\rangle}

\newcommand{\mabfa}{{\boldsymbol a}}
\newcommand{\mabfb}{{\boldsymbol b}}
\newcommand{\mabfj}{{\boldsymbol j}}
\newcommand{\mabfk}{{\boldsymbol k}}

\newcommand{\rr}[1]{\mathbf R^{#1}}

\newcommand{\zz}[1]{\mathbf Z^{#1}}
\newcommand{\nm}[2]{\Vert #1\Vert _{#2}}

\newcommand{\op}{\operatorname{Op}}

\newcommand{\sets}[2]{\{ \, #1\, ;\, #2\, \} }
\newcommand{\Sets}[2]{\left \{ \, #1\, ;\, #2\, \right \} }

\newcommand{\ep}{\varepsilon}
\newcommand{\fy}{\varphi}
\newcommand{\cdo}{\, \cdot \, }

\newcommand{\wpr}{{\text{\footnotesize $\#$}}}

\newcommand{\vrum}{\vspace{0.1cm}}

\newcommand{\Op}{\operatorname{Op}}

\newcommand{\GL}{\mathbf{M}}

\newcommand{\rd}{\mathbf{R} ^{d}}
\newcommand{\rdd}{\mathbf{R} ^{2d}}
\newcommand{\zd}{\zz {d}}

\newcommand{\dbar}{d\hspace*{-0.08em}\bar{}\hspace*{0.1em}}

\newcommand{\maclH}{\mathcal H}

\newcommand{\maclS}{\mathcal S}
\newcommand{\mascB}{\mathscr B}

\newcommand{\mascF}{\mathscr F}

\newcommand{\mascP}{\mathscr P}
\newcommand{\mascS}{\mathscr S}

\numberwithin{equation}{section}          

\newtheorem{thm}{Theorem}
\numberwithin{thm}{section}

\newcommand{\rubrik}{}
\newtheorem{prop}[thm]{Proposition}
\newtheorem{cor}[thm]{Corollary}
\newtheorem{lemma}[thm]{Lemma}

\theoremstyle{definition}

\newtheorem{defn}[thm]{Definition}

\theoremstyle{remark}
\newtheorem{rem}[thm]{Remark}

\author{Joachim Toft}

\address{Department of Mathematics,
Linn{\ae}us University, V{\"a}xj{\"o}, Sweden}

\email{joachim.toft@lnu.se}

\author{R{\"u}ya {\"U}ster}

\address{Department of Mathematics, Istanbul University, Istanbul, Turkey}

\title{Pseudo-differentialoperators on Orlicz modulation spaces}

\begin{document}

\begin{abstract}
We deduce continuity properties for pseudo-differential 
operators with symbols in quasi-Banach Orlicz modulation spaces 
when rely on other quasi-Banach Orlicz modulation spaces. 
In particular we extend certain results in \cite{GH1,GH2,Toft2,Toft10,Toft16}.
\end{abstract}

\keywords{Orlicz, quasi-Banach, quasi-Young functionals}

\subjclass[2010]{primary: 35S05, 46E30, 46A16, 42B35
secondary: 46F10}

\maketitle

\section{Introduction}\label{sec0}

\par


\par

In the paper we deduce continuity properties for
pseudo-differential operators when acting on quasi-Banach
Orlicz modulation spaces. For example, for a pseudo-differential
operator $\op (a)$ with the symbol $a$ we show that
the following is true:
\begin{itemize}
\item suppose that $q_0\in (0,1]$, $\Phi _j$ are quasi-Young
functions which satisfy $\Phi _j(t)\lesssim t^{q_0}$ near origin,
and that $a$ belongs to the classical
modulation space $M^{\infty ,q_0}(\rr {2d})$. Then
$\op (a)$ is continuous on the quasi-Banach Orlicz
modulation space $M^{\Phi _1,\Phi _2}(\rr d)$;

\vrum

\item suppose that $\Phi$ is a quasi-Young
function which satisfy $t\lesssim \Phi (t)$ near origin,
and that $a$ belongs to $M^{\Phi}(\rr {2d})$. Then
$\op (a)$ is continuous from $M^\infty (\rr d)$
to $M^{\Phi}(\rr d)$;

\vrum

\item suppose that $\Phi _0$ is a Young
function and $\Phi _0^*$  is the complementary
Young function,
and that $a$ belongs to $M^{\Phi _0}(\rr {2d})$. Then
$\op (a)$ is continuous from $M^{\Phi _0^*} (\rr d)$
to $M^{\Phi}(\rr d)$.
\end{itemize}

\noindent
(We refer to \cite{Ho1} and Section \ref{sec1}
for notations.)

\par

More generally, we deduce weighted versions
of such continuity results.
In particular we extend some continuity properties
for pseudo-differential operators when acting on
(ordinary) modulation spaces, e.{\,}g. in
\cite{CorNic,CorNic2,GH1,GH2,Toft2,Toft10,Toft16}.

\par

Essential parts of our analysis are based on
 \cite{SchF} by C. Schnackers
and H. F{\"u}hr concerning Orlicz modulation
spaces, and on \cite{ToUsNaOz} concerning
quasi-Banach Orlicz modulation spaces. In these
approaches, general properties and aspects
on quasi-Banach Orlicz spaces given in
\cite{HaH} by P. Harjulehto and P. H{\"a}st{\"o}
are fundamental. In this respect, we show that
for mixed quasi-Banach Orlicz modulation spaces
like $M^{\Phi _1,\Phi _2}_{(\omega )}(\rr {2d})$
we have $M^{\Phi ,\Phi}_{(\omega )}(\rr {2d})
=M^{\Phi}_{(\omega )}(\rr {2d})$ when
$\Phi$, $\Phi _1$ and $\Phi _2$ are quasi-Young
functions. This leads to convenient improvement
of the style of the continuity results for our
pseudo-differential operators when acting
on quasi-Banach Orlicz modulation spaces.

\medspace

In some situations it might be beneficial to replace 
Lebesgue norm estimates with more
refined Orlicz norm estimates. This may
appear when dealing with
certain non-linear
functionals. For example,
in statistics or statistical physics, the entropy applied on probability density
functions $f$ on $\rr d$ is given by
$$
\mathsf E(f) = -\int _{\rr d}f(x)\log f(x)\, dx.
$$
When investigating $\mathsf E$, it might be
more efficient
to replace the pair of Lebesgue spaces
$(L^1,L^\infty )$ by the pair of Orlicz spaces
$(L\log (L+1),L^{\cosh -1})$, where the Young functions
are given by
$$
\Phi (t) = t\log (1+t)
\quad \text{and}\quad \Phi (t) = \cosh(t)-1,
$$
respectively. We also observe that
the Zygmund space $L\log ^+L$
is an Orlicz space related to
Hardy-Littlewood maximal functions.
(See \cite{MajLab1,MajLab2} and the references 
therein.)

\par

Such questions are also relevant when
investigating localized Fourier
transforms like short-time
Fourier transforms $V_\phi f$ because
of the entropy conditions
$$
\mathsf E(|V_\phi f|^2)\ge C,
$$
for some constant $C$, when
$$
\nm f{L^2}=\nm g{L^2}=\nm \phi{L^2}=1.
$$
(See \cite{LieSol}.)
We remark that such refined Fourier transforms
are indispensable tools within 
time-frequency,
signal processing and certain parts of
quantum mechanics.

\par

In time-frequency analysis and signal processing,
non-stationary filters can be modelled by
pseudo-differential operators
$f\mapsto \op (a)f$, where
the symbols $a$ are determined by
time and frequency varying
filters, the target functions $f$ are
the original signals and $\op (a)f$
are the reflected signals. In such
situations it is suitable to discuss
continuity properties by means of
certain types of time-frequency
invariant (quasi-)Banach spaces. This leads
to modulation spaces.

\par

The classical
modulation spaces is a family of
function and distribution spaces, introduced
by Feichtinger in \cite{F1}. Here the
modulation spaces are defined by imposing
a weighted mixed Lebesgue norm estimate
on the short-time Fourier transforms of
the involved functions and distributions.
The theory
has thereafter been extended and generalized,
especially by Feichtinger and Gr{\"o}chenig
in \cite{FeiGro1,FeiGro2}, where the theory
of (Banach) modulation spaces was put
into the context of coorbit space theory.
A less abstract extension of the
classical modulation spaces is performed
in \cite{Fei5}, where Feichtinger
replace the mixed Lebesgue norm estimates
in \cite{F1} with more general translation
invariant norms to solid
Banach function spaces.

\par

Some extensions to the quasi-Banach case
have thereafter been performed in
e.{\,}g. \cite{GaSa,Rau1,Rau2,Toft15,Toft18}.

\par

In \cite{SchF}, F{\"u}hr and Schnacker
study Orlicz modulation spaces of the form
$M^{\Phi _1,\Phi _2}$, where $\Phi _1$ and
$\Phi _2$ are Young functions. That is,
they consider modulation spaces in \cite{Fei5},
where the solid Banach
function spaces are Orlicz spaces,
a naturally generalization of
$L^p$ spaces which contain certain Sobolev spaces
as subspaces. In
particular their investigations also
include the classical modulation spaces in
\cite{F1}, since these spaces are obtained
by choosing
$$
\Phi _j(t) = t^p
\quad \text{or}\quad
\Phi _j(t) =
\begin{cases}
0, & t\le 1,
\\[1ex]
\infty , & t>1.
\end{cases}
$$

\par

The analysis in \cite{SchF} is extended in 
\cite{ToUsNaOz} to quasi-Banach weighted
Orlicz modulation spaces,
$M^{\Phi _1,\Phi _2}_{(\omega)}(\rd )$,
where $\Phi _1$, $\Phi _2$ are quasi-Young 
functions of certain degrees
and $\omega$ is a suitable weight
function on $\rr {2d}$. In particular,
it is here allowed to let $\Phi _j(t)=t^p$
for every $p>0$ (instead of $p\ge 1$ as
in \cite{SchF}), which implies that
any modulation space
$M^{p,q}_{(\omega )}(\rr d)$
for $p,q\in (0,\infty $
are included in the studies in 
\cite{ToUsNaOz}.

\par

In the paper, our deduced continuity
for pseudo-differential operators,
are based on the various properties of
quasi-Banach Orlicz modulation spaces,
obtained in \cite{ToUsNaOz}.

\par

\section{Preliminaries}\label{sec1}

\par

In this section we recall some facts for Gelfand-Shilov spaces, Orlicz
spaces, Orlicz modulation spaces and pseudo-differential operators. First we
discuss some useful properties of Gelfand-Shilov spaces. Thereafter we
recall some classes of weight functions which are used later on
in the definition of Orlicz modulation spaces. Then Orlicz spaces
are In Subsections \ref{subsec1.3} and \ref{subsec1.4} we define and
present some properties for Orlicz spaces and Orlicz modulation spaces.
We conclude the section by discussing Gabor analysis for Orlicz
modulation spaces and pseudo-differential operators.

\par

\subsection{Gelfand-Shilov spaces}\label{subsec1.1}
We start by discussing Gelfand-Shilov spaces and their properties.
Let $0<s\in \mathbf R$ be fixed. Then the Gelfand-Shilov
pace $\mathcal S_{s}(\rr d)$
($\Sigma _{s}(\rr d)$) of Roumieu type (Beurling type) with parameter $s$
consists of all $f\in C^\infty (\rr d)$ such that
\begin{equation}\label{gfseminorm}
\nm f{\mathcal S_{s,h}}\equiv \sup \frac {|x^\beta \partial ^\alpha
f(x)|}{h^{|\alpha  + \beta |}(\alpha ! \beta !)^s}
\end{equation}
 is finite for some $h>0$ (for every $h>0$). Here the supremum should be taken
 over all $\alpha ,\beta \in \mathbf N^d$ and $x\in \rr d$. We equip
 $\mathcal S_{s}(\rr d)$ ($\Sigma _{s}(\rr d)$) by the canonical inductive limit
 topology (projective limit topology) with espect to $h>0$, induced by
 the semi-norms in \eqref{gfseminorm}.

 \par

 For any $s,s_0>0$ such that $1/2 \leq s_0<s$ we have
\begin{equation}\label{GSembeddings}
 \begin{alignedat}{3}
 \maclS _{s_0}(\rr d)
 &\hookrightarrow &
 \Sigma _{s}(\rr d)
 &\hookrightarrow &
 \maclS _s(\rr d)
 &\hookrightarrow
 \mascS (\rr d),
  \\[1ex]
 \mascS '(\rr d)
 &\hookrightarrow &
 \maclS _s' (\rr d)
 &\hookrightarrow &
 \Sigma _{s}'(\rr d)
 &\hookrightarrow
 \maclS _{s_0}'(\rr d),
 \end{alignedat}
 \end{equation}
 with dense embeddings.
 Here $A\hookrightarrow B$ means that
 the topological spaces $A$ and $B$ satisfy $A\subseteq B$ with
 continuous embeddings.
 The space $\Sigma _s(\rr d)$ is a Fr{\'e}chet space
 with seminorms $\nm \cdot {\mathcal S_{s,h}}$, $h>0$. Moreover,
 $\Sigma _s(\rr d)\neq \{ 0\}$, if and only if $s>1/2$, and
 $\maclS _s(\rr d)\neq \{ 0\}$, if and only
 if $s\ge 1/2$.

 \medspace

 The \emph{Gelfand-Shilov distribution spaces} $\mathcal S_{s}'(\rr d)$
and $\Sigma _s'(\rr d)$ are the dual spaces of $\mathcal S_{s}(\rr d)$ and
$\Sigma _s(\rr d)$, respectively.  As for the Gelfand-Shilov spaces there
 is a canonical projective limit topology (inductive limit topology) for $\maclS _{s}'(\rr d)$
 ($\Sigma _s'(\rr d)$). (Cf. \cite{GS,Pil1,Pil3}.)

 \par

 From now on we let $\mathscr F$ be the Fourier transform which
 takes the form
 $$
 (\mathscr Ff)(\xi )= \widehat f(\xi ) \equiv (2\pi )^{-\frac d2}\int _{\rr
 {d}} f(x)e^{-i\scal  x\xi }\, dx
 $$
 when $f\in L^1(\rr d)$. Here $\scal \cdo \cdo$ denotes the usual
 scalar product on $\rr d$. The map $\mathscr F$ extends
 uniquely to homeomorphisms on $\mathscr S'(\rr d)$,
 from $\mathcal S_s'(\rr d)$ to $\mathcal S_s'(\rr d)$ and
 from $\Sigma _s'(\rr d)$ to $\Sigma _s'(\rr d)$. Furthermore,
 $\mascF$ restricts to
 homeomorphisms on $\mathscr S(\rr d)$, from
 $\mathcal S_s(\rr d)$ to $\mathcal S_s(\rr d)$ and
 from $\Sigma _s(\rr d)$ to $\Sigma _s(\rr d)$,
 and to a unitary operator on $L^2(\rr d)$.

 \par

Gelfand-Shilov spaces can in convenient
 ways be characterized in terms of estimates of the functions and their Fourier
 transforms. More precisely, in \cite{ChuChuKim} it is proved that
 if $f\in \mascS '(\rr d)$ and $s>0$, then $f\in \maclS _s(\rr d)$
 ($f\in \Sigma _s(\rr d)$), if and only if
 \begin{equation}\label{Eq:GSFtransfChar}
 |f(x)|\lesssim e^{-r|x|^{\frac 1s}}
 \quad \text{and}\quad
 |\widehat f(\xi )|\lesssim e^{-r|\xi |^{\frac 1s}},
 \end{equation}
 for some $r>0$ (for every $r>0$).
 Here $r_1(\theta ) \lesssim r_2(\theta )$ means that $r_1(\theta ) \le c \cdot  r_2(\theta )$
 holds uniformly for all $\theta$
 in the intersection of the domains of $r_1$ and $r_2$
 for some constant $c>0$, and we
 write $r_1\asymp r_2$
 when $r_1\lesssim r_2 \lesssim r_1$.

 \par

 Gelfand-Shilov spaces and their distribution spaces can also
 be characterized by estimates of short-time Fourier
 transforms, (see e.{\,}g. \cite{GZ,Toft18}).
 More precisely, let $\phi \in \maclS _s (\rr d)$ be fixed. Then the \emph{short-time
 Fourier transform} $V_\phi f$ of $f\in \maclS _s '
 (\rr d)$ with respect to the \emph{window function} $\phi$ is
 the Gelfand-Shilov distribution on $\rr {2d}$, defined by
 \begin{equation} \label{Eq:ShorttimeF}
      V_\phi f(x,\xi )  =
 \mascF (f \, \overline {\phi (\cdo -x)})(\xi ).
 \end{equation}
 If $f ,\phi \in \maclS _s (\rr d)$, then it follows that
 $$
 V_\phi f(x,\xi ) = (2\pi )^{-\frac d2}\int f(y)\overline {\phi
 (y-x)}e^{-i\scal y\xi}\, dy .
 $$

\par

\subsection{Weight functions}\label{subsec1.2}

\par

A \emph{weight} or \emph{weight function} on $\rr d$ is a
positive function $\omega
\in  L^\infty _{loc}(\rr d)$ such that $1/\omega \in  L^\infty _{loc}(\rr d)$.
The weight $\omega$ is called \emph{moderate},
if there is a positive weight $v$ on $\rr d$ such that
\begin{equation}\label{moderate}
\omega (x+y) \lesssim \omega (x)v(y),\qquad x,y\in \rr d.
\end{equation}
If $\omega$ and $v$ are weights on $\rr d$ such that
\eqref{moderate} holds, then $\omega$ is also called
\emph{$v$-moderate}.
We note that \eqref{moderate}
implies that $\omega$ fulfills
the estimates
\begin{equation}\label{moderateconseq}
v(-x)^{-1}\lesssim \omega (x)\lesssim v(x),\quad x\in \rr d.
\end{equation}
We let $\mascP _E(\rr d)$ be the set of all moderate weights on $\rr d$.

\par

It can be proved that if $\omega \in \mascP _E(\rr d)$, then
$\omega$ is $v$-moderate for some $v(x) = e^{r|x|}$, provided the
positive constant $r$ is large enough (cf. \cite{Gc2.5}). That is,
\eqref{moderate} implies
\begin{equation}\label{Eq:weight0}
\omega (x+y) \lesssim \omega(x) e^{r|y|}
\end{equation}
for some $r>0$. In particular, \eqref{moderateconseq} shows that
for any $\omega \in \mascP_E(\rr d)$, there is a constant $r>0$ such that
$$
e^{-r|x|}\lesssim \omega (x)\lesssim e^{r|x|},\quad x\in \rr d.
$$

\par

We say that $v$ is
\emph{submultiplicative} if $v$ is even and
\eqref{moderate}
holds with $\omega =v$. In the sequel, $v$ and $v_j$ for
$j\ge 0$, always stand for submultiplicative weights if
nothing else is stated.

\par

We let $\mascP ^{0} _E(\rd)$ be the set of all $\omega\in \mascP _E(\rd)$
such that \eqref{Eq:weight0} holds for every $r>0$. We also let $\mascP (\rd)$
be the set of all $\omega\in \mascP _E(\rd)$ such that
$$
\omega (x+y) \lesssim \omega(x) (1+|y|)^r
$$
for some $r>0$.
Evidently,
$$
\mascP (\rd) \subseteq \mascP ^{0} _E(\rd) \subseteq \mascP _E(\rd).
$$

\par

\subsection{Orlicz Spaces}\label{subsec1.3}

\par

In this subsection we provide an overview
of some basic definitions and state some
technical results that will be needed.

\par

First we recall some facts concerning
Young functions and Orlicz
spaces. (See \cite{Rao,HaH}.)

\par

\begin{defn}\label{convex f.}
A function $\Phi :\mathbf R \rightarrow
\mathbf R \cup \{ \infty\}$ is called \emph{convex} if
\begin{equation*}
\Phi (s_1 t_1+ s_2 t_2)
\leq s_1 \Phi (t_1)+s_2\Phi (t_2)
\end{equation*}
when
$s_j,t_j\in \mathbf{R}$
satisfy $s_j \ge 0$ and
$s_1 + s_2 = 1,\ j=1,2$.
\end{defn}
We observe that $\Phi$ might not
 be continuous, because we permit
 $\infty$ as function value. For example,
$$
\Phi (t)=
\begin{cases}
  c,&\text{when}\ t \leq a
  \\[1ex]
   \infty ,&\text{when}\ t>a
\end{cases}
$$
is convex but discontinuous at $t=a$.

\par

\begin{defn}\label{Young func.}
Let $r_0\in(0,1]$, $\Phi _0$ and
$\Phi$ be functions from $[0,\infty)$ to $[0,\infty]$.
Then $\Phi _0$ is called a \emph{Young function} if
\begin{enumerate}
  \item $\Phi _0$ is convex,

  \vrum

  \item $\Phi _0(0)=0$,

  \vrum

  \item $\lim
\limits_{t\rightarrow\infty} \Phi _0(t)=+\infty$.
\end{enumerate}
The function $\Phi$ is called
\emph{$r_0$-Young function} or
\emph{quasi-Young function of
order $r_0$}, if $\Phi (t)=\Phi _0 (t^{r_0})$, $t \geq 0$,
for some Young function $\Phi _0$.
\end{defn}

\par

It is clear that $\Phi$ in
Definition \ref{Young func.} is
non-decreasing, because if $0\leq t_1\leq t_2$
and $s\in [0,1]$ is chosen such
that $t_1=st_2$ and $\Phi _0$ is the same as in
Definition \ref{Young func.}, then
\begin{equation*}
    \Phi (t_1)=\Phi _0(s^{r_0}t_2^{r_0}+(1-s^{r_0})0)
    \leq s^{r_0}\Phi _0(t_2^{r_0})+(1-s^{r_0})\Phi _0(0)
    \leq \Phi (t_2),
\end{equation*}
since $\Phi (0) = \Phi _0(0)=0$ and $s\in [0,1]$.

\par

\begin{defn}
Let $(\Omega,\Sigma,\mu)$ be a Borel measure
space, with $\Omega \subseteq \rd$, $\Phi _0$ be a Young function and
let $\omega_0 \in \mascP _E(\rr d)$.
\begin{enumerate}
    \item $L^{\Phi _0}_{(\omega_0)}(\mu)$ consists
    of all $\mu$-measurable functions
    $f:\Omega \rightarrow
    \mathbf C$ such that
    $$
    \Vert f\Vert_{L^{\Phi _0}_{(\omega_0)}(\mu)}=\inf  \Sets{\lambda>0}{\int_\Omega \Phi _0
    \left (
    \frac{|f(x) \cdot \omega_0 (x)|}{\lambda}
    \right )
    d\mu (x)\leq 1}
    $$
is finite. Here $f$ and $g$ in $L^{\Phi _0}_{(\omega_0)}(\mu)$ are equivalent if $f=g$ a.e.

\vrum

\item Let $\Phi$ be a quasi-Young
function of order $r_0\in (0,1]$,
given by $\Phi (t)=\Phi _0(t^{r_0})$,
$t\geq 0$, for some Young function $\Phi _0$. Then $L^ \Phi _{(\omega_0)}(\mu)$ consists
of all $\mu$-measurable functions
$f:\Omega \rightarrow \mathbf C$ such that
$$
\Vert f\Vert_{L^{\Phi}_{(\omega_0)}(\mu)}
=
(\Vert|f \cdot \omega_0 |^{r_0}\Vert_{L^{\Phi _0}(\mu)})^{1/r_0}
$$
is finite.
\end{enumerate}
\end{defn}

\par

\begin{rem}
Let $\Phi$, $\Phi _0$ and $\omega_0$ be the
same as in Definition \ref{Young func.}.
Then it follows by straight-forward computation that
$$
\Vert f\Vert_{L^\Phi _{(\omega_0)}(\mu )}=\inf\Sets{\lambda>0}
{\int_\Omega \Phi _0
    \left (
\frac{|f(x) \cdot \omega_0 (x)|^{r_0}}{\lambda^{r_0}}
\right)
d\mu(x)\leq 1}.
$$
\end{rem}

\par

\begin{defn}\label{d1}
Let $(\Omega_j ,\Sigma_j ,\mu_j )$ be Borel
measure spaces, with $\Omega_j \subseteq \rd$, $r_0\in (0,1]$, $\Phi _{j}$ be
$r_0$-Young functions, $j=1,2$ and let $\omega \in \mascP _E (\rdd)$. Then
the mixed quasi-norm Orlicz space ${L^{\Phi _1, \Phi _2}_{(\omega)}}
= {L^{\Phi _1, \Phi _2}_{(\omega)}}(\mu_1 \otimes \mu_2)$ consists of all $\mu_1 \otimes
\mu_2$-measurable functions $f:\Omega_1 \times \Omega_2 \rightarrow
\mathbf C$ such that
$$
\Vert f\Vert_{L^{\Phi _1, \Phi _2}_{(\omega)}} \equiv
\Vert f_{1,\omega}\Vert_{L^{\Phi _2}},
$$
is finite, where
$$
f_{1,\omega}(x_2)=\Vert f(\cdo ,x_2) \omega(\cdo, x_2)\Vert_{L^{\Phi _{1}}}.
$$
\end{defn}

\par

If $r_0=1$ in Definition \ref{d1}, then $L^{\Phi _1, \Phi _2}_{(\omega)}
(\mu_1 \otimes \mu_2)$ is a Banach space and is called a mixed norm Orlicz space.

\par

\begin{rem}
Suppose $\Phi _j$ are quasi-Young
functions of order $q_j\in (0,1]$, $j=1,2$.
Then both $\Phi _{1}$ and $\Phi _{2}$
are quasi-Young functions of order
$r_0=\min(q_1,q_2)$.
\end{rem}

\par

Let $\Lambda \subseteq \rr d$ be a lattice, i.{\,}e., $\Lambda$ is given
by
$$
\Lambda = \sets {n_1e_1+\cdots +n_de_d}{(n_1,\dots ,n_d)\in \zz d}
$$
for some basis $e_1,\dots ,e_d$ of $\rr d$. 
Then $\ell_0'(\Lambda )$ is the set of all
formal sequences
$$
\{ a(n)\} _{n\in \Lambda} =
\sets{a(n)}{n\in \Lambda}\subseteq \mathbf C,
$$
and let $\ell_0 (\Lambda )$ be the set of all
sequences $\{ a(n)\} _{n\in \Lambda}$ such that $a(n)\neq 0$ for at
most finite numbers of $n$. We observe that
$$
\Lambda ^2=\Lambda \times \Lambda
=
\sets {(x,\xi )}{x,\xi \in \Lambda}
$$
is a lattice in $\rr {2d}\simeq \rr d\times \rr d$.

\par

\begin{rem}
Let $\Lambda \subseteq \rr d$ be a lattice, $\Phi, \Phi _1$ and $\Phi _2$
be $r_0$-Young functions, $\omega_0, v_0
\in \mascP_E (\rd)$ and  $\omega, v \in \mascP_E (\rdd )$ be such that $\omega
_0$ and $\omega$ are $v_0$- respectively $v$-moderate. (In the sequel it is
understood that all lattices contain $0$.) Then we set
$$
L^{\Phi}_{(\omega_0)}(\rd) = L^{\Phi}_{(\omega_0)}(\mu)
\quad
\text{and}
\quad
L^{\Phi _1, \Phi _2}_{(\omega_0)}(\rdd ) =
L^{\Phi _1,\Phi _2}_{(\omega_0)}(\mu \otimes \mu),
$$
when $\mu$ is the Lebesgue measure on $\rr d$.
If instead
$\mu$ is the standard (Haar) measure on $\Lambda$, i.e.
$\mu(n)=1,\ n\in \Lambda$, and
$$
\ell^{\Phi}_{(\omega)}(\Lambda ) = \ell^{\Phi}_{(\omega)}(\mu)
\quad
\text{and}
\quad
\ell^{\Phi _1, \Phi _2}_{(\omega)}(\Lambda \times \Lambda )
= \ell^{\Phi _1,\Phi _2}_{(\omega)}(\mu \otimes \mu).
$$
Evidently, $\ell^{\Phi _1, \Phi _2}
_{(\omega)}(\Lambda \times \Lambda )\subseteq
\ell _0'(\Lambda \times \Lambda )$.
\end{rem}

\par

\begin{lemma}\label{T}
Let $\Phi , \Phi _j$ be Young functions, $j=1,2$, $\omega_0, v_0 \in \mascP_E (\rd)$
and $\omega, v \in \mascP_E (\rdd)$ be such that $\omega_0$ is $v_0$-moderate
and $\omega$ is $v$-moderate. Then $L^{\Phi}_{(\omega_0)}(\rd)$ and
$L^{\Phi _{1}, \Phi _{2}}_{(\omega)}(\rr{2d})$ are
 invariant under translations, and
$$
\Vert f(\cdo - x)\Vert_{L^\Phi _{(\omega_0)}} \lesssim
\Vert f\Vert_{L^\Phi _{(\omega_0)}} v_0(x), \quad f\in L^\Phi _{(\omega_0)}(\rd),\ x\in \rd\;,
$$
and
$$
\Vert f(\cdo - (x,\xi))\Vert_{L^{\Phi _{1}, \Phi _{2}}_{(\omega )}}
\lesssim
\Vert f\Vert_{L^{\Phi _{1}, \Phi _{2}}_{(\omega)}}v(x,\xi ),
\quad f\in L^{\Phi _{1}, \Phi _{2}}_{(\omega )} (\rdd ),\ (x,\xi ) \in \rdd.
$$
\end{lemma}

\par

\begin{proof}
We only prove the assertion for $L^{\Phi _{1}, \Phi _{2}}_{(\omega )}(\rr{2d})$.
The other part follows by similar arguments and is left for the reader.

\par

We have $\Phi _j(t) = \Phi _{0,j}(t^{r_0}),\ t\geq 0$, for some Young functions
$\Phi _{0,j}$, $j=1,2$. This gives
\begin{multline*}
\nm {f(\cdo - (x,\xi ))}{L^{\Phi _{1}, \Phi _{2}}_{(\omega )}}=
\left(
\nm{|f(\cdo - (x,\xi ))\omega |^{r_0}}{L^{\Phi _{0,1}, \Phi _{0,2}}}
\right)^{\frac{1}{r_0}}
\\[1ex]
\lesssim
\left(
\nm{|f(\cdo - (x,\xi ))\omega(\cdo - (x,\xi ))v(x,\xi )|^{r_0}}{L^{\Phi _{0,1}, \Phi _{0,2}}}
\right)^{\frac{1}{r_0}}
\\[1ex]
=
\left (
\nm{|f\cdot \omega |^{r_0}}{L^{\Phi _{0,1}, \Phi _{0,2}}}
\right )^{\frac{1}{r_0}}
\cdot v(x,\xi )=
\nm{f}{L^{\Phi _{1}, \Phi _{2}}_{(\omega )}}\cdot
v(x,\xi ).
\end{multline*}
Here the inequality follows from the fact that $\omega$ is $v$-moderate, and the last two
relations follow from the definitions.
\end{proof}

We refer to \cite{SchF,Rao,HaH} for more facts about Orlicz spaces.

\par

\subsection{Orlicz modulation spaces}\label{subsec1.4}

\par

Before considering Orlicz modulation spaces, we recall the definition of classical
modulation spaces. (Cf. \cite{F1,Fei5}.)

\par

\begin{defn}\label{Def:Orliczmod}
Let $\phi(x) = \pi ^{-\frac{d}{4}}e^{-\frac{|x|^2}{2}},\ x\in \rd$, $p,q\in (0,\infty]$ and
$\omega $ be a weight on $\rdd$.
Then the \emph{modulation spaces} $M^{p,q}_{(\omega)}(\rd)$
is set of all $f\in \maclS _{1/2}'
(\rr d)$ such that $V_\phi f\in L^{p,q}_{(\omega)}(\rr {2d})$.
We equip these spaces with the quasi-norm
\begin{equation*}
\nm f{M^{p,q}_{(\omega)}} \equiv \nm {V_\phi f}{L^{p,q}_{(\omega)}}.
\end{equation*}
Also let $\Phi , \Phi _1,\Phi _2$ be quasi-Young functions.
Then the \emph{Orlicz modulation spaces}
$M^{\Phi}_{(\omega )} (\rd )$ and  $M^{\Phi _{1}, \Phi _{2}}_{(\omega )}(\rd )$ are given by
\begin{equation}\label{Eq:Orliczmod1}
M^{\Phi}_{(\omega )}(\rd )=
\sets{f \in \maclH' _{\flat}(\rd )} {V_\phi f\in
L^{\Phi}_{(\omega )} (\rdd )}
\end{equation}
and
\begin{equation}\label{Eq:Orliczmod2}
M^{\Phi _{1}, \Phi _{2}}_{(\omega )}(\rd )=
\sets{f\in \maclH '_{\flat}(\rd )} { V_\phi f\in
L^{\Phi _{1}, \Phi _{2}} _{(\omega )}(\rdd )}.
\end{equation}
The quasi-norms on $M^{\Phi}_{(\omega )}(\rd )$ and
$M^{\Phi _{1}, \Phi _{2}}_{(\omega )}(\rd )$ are given by
\begin{equation}\label{NE}
\Vert f\Vert_{M^{\Phi}_{(\omega )}} =\Vert V_\phi f\Vert_{L^{\Phi}_{(\omega )}}
\end{equation}
and
\begin{equation}\label{MNE}
\nm f {M^{\Phi _{1}, \Phi _{2}}_{(\omega )}}
=\nm {V_\phi f} {L^{\Phi _{1}, \Phi _{2}}_{(\omega)}}.
\end{equation}
\end{defn}

\par

For conveniency we set
$$
M^{p,q}=M^{p,q}_{(\omega )},
\quad
M^\Phi = M^\Phi _{(\omega )}
\quad \text{and}\quad
M^{\Phi _1,\Phi _2} = M^{\Phi _1,\Phi _2} _{(\omega )}
\quad \text{when}\quad
\omega (x,\xi)=1,
$$
and $M^p=M^{p,p}$ and $M^p_{(\omega)}=M^{p,p}_{(\omega)}$.

\par

We notice that \eqref{NE} and \eqref{MNE} are norms when
$\Phi , \Phi _1$ and $\Phi _2$ are Young functions.
If $\omega \in \mascP_E(\rdd)$ as in Definition \ref{Def:Orliczmod},
then we prove later on that the conditions
$$
\nm {V_\phi f} {L^{\Phi _{1}, \Phi _{2}}_{(\omega)}} <\infty
\quad \text{and}\quad
\nm {V_\phi f} {L^{\Phi}_{(\omega)}}<\infty
$$
are independent of the choices of $\phi$ in $\Sigma_1(\rd)\setminus{0}$ and that
different $\phi$ give rise to equivalent quasi-norms.

\par

Later on we need the following proposition.

 \par

\begin{prop}\label{subsets}
Let $\Phi, \Phi _j$ be Young functions, $j=1,2$, $\omega_0 \in \mascP_E (\rd )$
and $\omega \in \mascP_E (\rdd )$. Then
$$
\mascS (\rd )\subseteq L^{\Phi} (\rd )\subseteq \mascS '(\rd), \quad
\mascS (\rdd )\subseteq L^{\Phi _{1},\Phi _{2}}(\rdd)\subseteq \mascS '(\rdd ),
$$
$$
\Sigma_1 (\rd ) \subseteq L^{\Phi}_{(\omega _0)} (\rd )
\subseteq \Sigma_1 ' (\rd ), \quad \Sigma _1 (\rdd )
\subseteq
L^{\Phi _{1},\Phi _{2}}_{(\omega )}(\rdd )
\subseteq \Sigma_1 ' (\rdd ).
$$
\end{prop}

\par

\begin{proof}
 Let $v_0 \in \mascP _E(\rd)$ and $v\in \mascP _E(\rr {2d})$ be chosen
 such that $\omega_0$ is $v_0$-moderate and $\omega$ is $v$-moderate.
 Since $L^{\Phi}_{(\omega _0)}(\rd )$ and
 $L^{\Phi _{1}, \Phi _{2}}_{(\omega )}(\rdd )$
 are invariant under translation and modulation, we have
 $$
 M^1_{(v_0)}(\rd ) \subseteq
L^{\Phi}_{(\omega _0)}(\rd ) \subseteq M^{\infty}_{(1/v_0)}(\rd ),
$$
and
$$
M^1_{(v)}(\rdd ) \subseteq  L^{\Phi _{1},\Phi _{2}}_{(\omega )}(\rdd )
\subseteq M^{\infty}_{(1/v)}(\rdd ).
$$
(see \cite{Gc2,Toft19}). The result now follows from well-known inclusions between modulation spaces, Schwartz spaces, Gelfand-Shilov spaces, and their duals.
\end{proof}

\par

The next result gives some information about the roles that $\Phi _1$ and
$\Phi _2$ play for $M^{\Phi _{1},\Phi _{2}}$. We
omit the proof since it can be found in \cite{ToUsNaOz}. See also
\cite{SchF} for the Banach case.

\par

\begin{prop}\label{Prop:OrliczModInvariance}
Let $\Phi _j$ and $\Psi _j$, $j=1,2$, be quasi-Young functions,
$\Lambda$ be a lattice in $\rr d$ and $\omega \in \mascP_E(\rdd )$.
Then the following conditions are equivalent:
\begin{enumerate}
    \item $M^{\Phi _{1}, \Phi _{2}}_{(\omega )}(\rd )\subseteq
    M^{\Psi _{1} ,\Psi _{2}}_{(\omega )}(\rd )$;

    \vrum

    \item $\ell^{\Phi _{1}, \Phi _{2}}_{(\omega )}(\Lambda )\subseteq
    \ell^{\Psi _{1}, \Psi _{2}}_{(\omega )}(\Lambda )$;

    \vrum

    \item $\Psi _{j} (t)\lesssim  \Phi _{j} (t)$
    for every $t\in [0, t_0]$, for some $t_0>0$.
\end{enumerate}
\end{prop}

\par

\subsection{Gabor frames}\label{subsec1.5}

\par


\par

\begin{defn}
Let $\omega, v \in \mascP_E (\rdd )$ be such that $\omega$ is
$v$-moderate, $\phi , \psi \in M^1_{(v)}(\rd )$, $\ep >0$ and let
$\Lambda \subseteq \rd$ be a lattice.
\begin{enumerate}
    \item  The \emph{analysis operator} $C_{\phi}^{\ep , \Lambda}$
    is the operator from $M^\infty_{(\omega)}(\rd)$ to $\ell ^\infty _{(\omega)}
    (\ep \Lambda ^2)$, given by
$$
C_{\phi}^{\ep , \Lambda} f
\equiv
\{ V_\phi f( j,\iota)\} _{j,\iota \in \ep \Lambda}.
$$
\item  The \emph{synthesis operator} $D_
\psi^{\ep , \Lambda}$ is the operator from $\ell^\infty_{(\omega)}(\ep \Lambda ^2)$
to $M^\infty_{(\omega)}(\rd)$, given by
$$
D_{\psi}^{\ep , \Lambda} c
\equiv
\sum_{j,\iota \in \ep \Lambda'}
c(j,\iota)e^{i \scal \cdo \iota}\psi (\cdo- j).
$$
\item  The \emph{Gabor frame operator} $S_{\phi ,\psi}^{\ep , \Lambda}$
is the operator on $M^\infty_{(\omega)}(\rd )$, given by
$D_{\psi}^{\ep , \Lambda} \circ C_{\phi}^{\ep , \Lambda},$ i.e.
$$
S_{\phi,\psi}^{\ep , \Lambda}
f \equiv
\sum_{j,\iota \in \ep \Lambda'}
V_\phi f(j,\iota)
e^{i\scal \cdo \iota}\psi (\cdo- j).
$$
\end{enumerate}
\end{defn}

\par

The next result shows that it is possible to find suitable $\phi$ and
$\psi$ in the previous definition.

\par

\begin{lemma}\label{Lemma:GoodFrames0}
Let $\Lambda \subseteq \rd$ be a lattice,
$v\in \mascP _E(\rr {2d})$ be submultiplicative and
$\phi \in M^1_{(v)}(\rr d) \setminus{0}$.
Then there is an $\ep >0$ and $\psi \in M^1_{(v)}(\rr d) \setminus{0}$ such that
\begin{equation}\label{Eq:GoodFrames0}
 \{ \phi (x-j)e^{i\scal x\iota } \} _{j,\iota \in \ep \Lambda}
\quad \text{and}\quad
 \{ \psi (x-j)e^{i\scal x\iota } \} _{j,\iota \in \ep \Lambda}
\end{equation}
are dual frames to each others.
\end{lemma}

\par

\begin{rem}\label{Remark:GoodFrames0}
There are several ways to achieve dual frames \eqref{Eq:GoodFrames0}.
In fact, let $v, v_0\in \mascP_E(\rdd )$ be submultiplicative such that $\omega$ is
$v$-moderate and $L^1_{(v_0)}(\rdd )\subseteq L^r(\rdd ),\ r\in(0,1]$. Then Lemma
\ref{Lemma:GoodFrames0} guarantees that for some choice of $\phi, \psi
\in M^1_{(v_0 v)}(\rd)\subseteq M^r_{(v)}(\rd)$ and lattice $\Lambda$
, the set in \eqref{Eq:GoodFrames0} are dual frames to each other, and that
$\psi = (S^\Lambda_{\phi ,\phi})^{-1}\phi$. (Cf.
\cite[Proposition 1.5 and Remark 1.6]{Toft16}.)
\end{rem}

\par

\begin{lemma}\label{Lemma:GoodFrames}
Let $\Lambda \subseteq \rd$ be a lattice, $v\in \mascP _E(\rr {4d})$
be submultiplicative
$\phi _1,\phi _2 \in \Sigma_1(\rd)\setminus{0}$ and
\begin{equation*}
\varphi (x,\xi) = \phi _1(x)\overline{\widehat{\phi}_{2}(\xi)}e^{-i\scal {x}{\xi}}.
\end{equation*}
Then there is an $\ep >0$ such that
\begin{equation*}
\{\varphi(x-j,\xi-\iota)e^{i(\scal x \kappa+ \scal k \xi)} \} _{j,\iota ,k,\kappa
\in \ep \Lambda}
\end{equation*}
is a Gabor frame with canonical dual frame
\begin{equation*}
\{\psi(x-j,\xi-\iota)e^{i(\scal x \kappa + \scal k \xi)}\} _{j,\iota ,k,\kappa 
\in \ep \Lambda}
\end{equation*}
where $\psi = (S_{\varphi,\varphi}^{\Lambda^2\times\Lambda^2})^{-1} \varphi$
belongs to $M_{(v)}^r(\rr {2d})$ for every $r>0$.
\end{lemma}

\par

The next result shows that Gabor theory is suitable when
dealing with Orlicz modulation spaces. We
omit the proof since the result follows from
\cite[Theorem 4.7]{ToUsNaOz}. See also \cite{SchF} for the Banach case.

\par

\begin{prop}\label{Prop:Gabor}
Let $\Lambda \subseteq \rd$ be a lattice, $v\in \mascP _E(\rr {4d})$
be submultiplicative $\Phi _1,\Phi _2$ be quasi-Young functions of order
$r_0 \in(0,1]$, $\omega, v \in \mascP_E (\rdd )$ be such that $\omega$
is $v$-moderate and let $\phi, \psi\in M^{r_0}_{(v)}(\rd )$ and $\ep >0$
be chosen such that
\begin{equation}\label{Eq:DualFrames}
\{ e^{i\scal \cdo \kappa}\phi (\cdo -k) \} _{k,\kappa \in \ep \Lambda}
\quad \text{and}\quad
\{ e^{i\scal \cdo \kappa}\psi (\cdo -k) \} _{k,\kappa \in \ep \Lambda}
\end{equation}
are dual frames to each others. If
$f\in M_{(\omega)}^{\Phi _1, \Phi _2}(\rd)$, then
\begin{align*}
f&=\sum_{k,\kappa \in \ep \Lambda}
(V_\psi f)(k,\kappa )
e^{i\scal \cdo \kappa }\phi(\cdo - k)
\end{align*}
with unconditionally convergence in $M^{\Phi _1,\Phi _2}_{(\omega)}(\rr d)$
when
$\mascS(\rdd )$ is dense in $L^{\Phi _1,\Phi _2}(\rdd )$,
and with convergence in $M^\infty_{(\omega)}(\rd)$ with respect to
the weak$^*$ topology otherwise. It holds
\begin{multline}\label{Eq:EquivModNormDisc}
\Vert \{ (V_\phi f)(k,\kappa )\}_{k,\kappa \in \ep \Lambda}
\Vert_{\ell_{(\omega)}^{\Phi _1, \Phi _2}}
\asymp
\Vert \{ (V_\psi f)(k,\kappa )\}_{k,\kappa\in \ep \Lambda}
\Vert_{\ell_{(\omega )}^{\Phi _1, \Phi _2}}
\\[1ex]
\asymp \Vert f \Vert_{M^{\Phi _1, \Phi _2}_{(\omega)}}.
\end{multline}
\end{prop}

\par

We also recall that the previous result was heavily based on the following
consequence of Theorems 4.5 and 4.6 in \cite{ToUsNaOz}. The proof is
therefore omitted.

\par

\begin{prop}\label{Prop:AnalysisSynthOp}
Let $\Lambda \subseteq \rd$ be a lattice, $\ep >0$, $\phi ,\psi \in \Sigma _1(\rr d)$,
$\Phi _1,\Phi _2$ be quasi-Young functions of order
$r_0 \in(0,1]$, and let $\omega, v \in
\mascP_E (\rdd )$ be such that $\omega$ is $v$-moderate.
Then the the following is true:
\begin{enumerate}
\item the analysis operator $C_{\phi}^{\ep ,\Lambda}$ is continuous from
$M^{\Phi _1,\Phi _2}_{(v)}(\rd)$ into $\ell^{\Phi _1,\Phi _2}_{(\omega)}(\ep \Lambda ^2)$, and
$$
\nm {C_{\phi}^{\ep ,\Lambda}f}{\ell^{\Phi _1,\Phi _2}_{(\omega)}}
\lesssim
\nm f{ M_{(\omega)}^{\Phi _1,\Phi _2}},
\quad
f\in M_{(\omega)}^{\Phi _1,\Phi _2}(\rd )\text;
$$

\vrum

\item the synthesis operator $D_{\psi}^{\ep ,\Lambda}$ is continuous from
$\ell^{\Phi _1,\Phi _2}_{(\omega)}(\ep \Lambda ^2)$ into
$M^{\Phi _1,\Phi _2}_{(\omega)}(\rd)$, and
$$
\Vert D_{\psi}^{\ep ,\Lambda}
c\Vert_{M^{\Phi _1,\Phi _2}_{(\omega)}}
\lesssim
\Vert c\Vert_{\ell^{\Phi _1,\Phi _2}_{(\omega)}},
\quad
c\in \ell^{\Phi _1,\Phi _2}_{(\omega)}(\ep \Lambda ^2).
$$
\end{enumerate}
\end{prop}

\par

\subsection{Pseudo-differential operators}\label{subsec1.6}

\par

Let $\GL (d,\Omega )$ be the set of all $d\times d$-matrices with
entries in the set $\Omega$, and let $s\ge 1/2$, $a\in \maclS _s
(\rr {2d})$ and $A\in \GL (d,\mathbf R)$ be fixed.
Then the pseudo-differential operator $\op _A(a)$ is the linear and
continuous operator on $\maclS _s (\rr d)$, given by
\begin{equation}\label{e0.5}
(\op _A(a)f)(x)
=
(2\pi  ) ^{-d}\iint a(x-A(x-y),\xi )f(y)e^{i\scal {x-y}\xi }\, dyd\xi ,
\end{equation}
when $f\in \maclS _s(\rr d)$. For
general $a\in \maclS _s'(\rr {2d})$, the
pseudo-differential operator $\op _A(a)$ is defined as the linear and
continuous operator from $\maclS _s(\rr d)$ to $\maclS _s'(\rr d)$ with
distribution kernel given by
\begin{equation}\label{atkernel}
K_{a,A}(x,y)=(2\pi )^{-d/2}(\mascF _2^{-1}a)(x-A(x-y),x-y).
\end{equation}
Here $\mascF _2F$ is the partial Fourier transform of $F(x,y)\in
\maclS _s'(\rr {2d})$ with respect to the $y$ variable. This
definition makes sense, since the mappings
\begin{equation}\label{homeoF2tmap}
\mascF _2\quad \text{and}\quad F(x,y)\mapsto F(x-A(x-y),x-y)
\end{equation}
are homeomorphisms on $\maclS _s'(\rr {2d})$.
In particular, the map $a\mapsto K_{a,A}$ is a homeomorphism on
$\maclS _s'(\rr {2d})$.

\par

An important special case appears when $A=t\cdot I$, with
$t\in \mathbf R$. Here and in what follows, $I\in \GL (d,\mathbf R)$ denotes
the $d\times d$ identity matrix. In this case we set
$$
\op _t(a) = \op _{t\cdot I}(a).
$$
The normal or Kohn-Nirenberg representation, $a(x,D)$, is obtained
when $t=0$, and the Weyl quantization, $\op ^w(a)$, is obtained
when $t=\frac 12$. That is,
$$
a(x,D) = \op _0(a)
\quad \text{and}\quad \op ^w(a) = \op _{1/2}(a).
$$

\par

For any $K\in \maclS '_s(\rr {d_1+d_2})$, we let $T_K$ be the
linear and continuous mapping from $\maclS _s(\rr {d_1})$
to $\maclS _s'(\rr {d_2})$, defined by the formula
\begin{equation}\label{pre(A.1)}
(T_Kf,g)_{L^2(\rr {d_2})} = (K,g\otimes \overline f )_{L^2(\rr {d_1+d_2})}.
\end{equation}
It is well-known that if $A\in \GL (d,\mathbf R)$, then it follows from Schwartz kernel
theorem that $K\mapsto T_K$ and $a\mapsto \op _A(a)$ are bijective
mappings from $\mascS '(\rr {2d})$
to the set of linear and continuous mappings from $\mascS (\rr d)$ to
$\mascS '(\rr d)$ (cf. e.{\,}g. \cite{Ho1}).

\par

Furthermore, by e.{\,}g. \cite[Theorem 2.2]{LozPerTask} it follows
that the same holds true if each $\mascS$ and $\mascS '$ are
replaced by $\maclS _s$ and $\maclS _s'$, respectively, or by
$\Sigma _s$ and $\Sigma _s'$, respectively.

\par

In particular, for every $a_1\in \maclS _s '(\rr {2d})$ and $A_1,A_2\in
\GL (d,\mathbf R)$, there is a unique $a_2\in \maclS _s '(\rr {2d})$ such that
$\op _{A_1}(a_1) = \op _{A_2} (a_2)$. The following result explains the
relations between $a_1$ and $a_2$.

\par

\begin{prop}\label{Prop:CalculiTransfer}
Let $a_1,a_2\in \maclS _{1/2}'(\rr {2d})$ and $A_1,A_2\in \GL (d,\mathbf R)$.
Then
\begin{equation}\label{calculitransform}
\op _{A_1}(a_1) = \op _{A_2}(a_2) \quad \Leftrightarrow \quad
e^{i\scal {A_2D_\xi}{D_x }}a_2(x,\xi )=e^{i\scal {A_1D_\xi}{D_x }}a_1(x,\xi ).
\end{equation}
\end{prop}

\par

In \cite{Toft15}, a proof of the previous proposition is given, which is similar to
the proof of the case $A=t\cdot I$ in \cite{Ho1,Sh,Tr}.

\par

Let $a\in \maclS _s '(\rr {2d})$ be
fixed. Then $a$ is called a rank-one element with respect to
$A\in \GL (d,\mathbf R)$, if $\op _A(a)$ is an operator of rank-one,
i.{\,}e.
\begin{equation}\label{trankone}
\op _A(a)f=(f,f_2)f_1, \qquad f\in \maclS _s(\rr d),
\end{equation}
for some $f_1,f_2\in \maclS _s '(\rr d)$. By
straight-forward computations it follows that \eqref{trankone}
is fulfilled if and only if $a=(2\pi
)^{\frac d2}W_{f_1,f_2}^A$, where $W_{f_1,f_2}^A$
is the $A$-Wigner distribution, defined by the formula
\begin{equation}\label{wignertdef}
W_{f_1,f_2}^A(x,\xi ) \equiv \mascF \big (f_1(x+A\cdo
)\overline{f_2(x+(A-I)\cdo )} \big ) (\xi ),
\end{equation}
which takes the form
$$
W_{f_1,f_2}^A(x,\xi ) =(2\pi )^{-\frac d2} \int
f_1(x+Ay)\overline{f_2(x+(A-I)y) }e^{-i\scal y\xi}\, dy,
$$
when $f_1,f_2\in \maclS _s (\rr d)$. By combining these facts
with \eqref{calculitransform}, it follows that
\begin{equation}\label{wignertransf}
e^{i\scal {A_2D_\xi  }{D_x}}W_{f_1,f_2}^{A_2}
=
e^{i\scal {A_1D_\xi }{D_x}} W_{f_1,f_2}^{A_1},
\end{equation}
for every $f_1,f_2\in \maclS _s '(\rr d)$ and $A_1,A_2\in \GL (d,\mathbf R)$. Since
the Weyl case is particularly important, we set
$W_{f_1,f_2}^{A}=W_{f_1,f_2}$ when $A=\frac 12I$, i.{\,}e.
$W_{f_1,f_2}$ is the usual (cross-)Wigner distribution of $f_1$ and
$f_2$.

\par

For future references we note the link
\begin{multline}\label{tWigpseudolink}
(\op _A(a)f,g)_{L^2(\rr d)} =(2\pi )^{-d/2}(a,W_{g,f}^A)_{L^2(\rr {2d})},
\\[1ex]
a\in \maclS _s'(\rr {2d}) \quad\text{and}\quad f,g\in \maclS _s(\rr d)
\end{multline}
between pseudo-differential operators and Wigner distributions,
which follows by straight-forward computations (see e.{\,}g.
\cite{Toft17} and the references therein).

\medspace

For any $A\in \GL (d,\mathbf R)$, the
$A$-product, $a\wpr _Ab$ between $a\in \maclS _s' (\rr {2d})$
and $b\in \maclS _s'(\rr {2d})$ is defined by the formula
\begin{equation}\label{wprtdef}
\op _A(a\wpr _A b) = \op _A(a)\circ \op _A(b),
\end{equation}
provided the right-hand side makes sense as a continuous operator from
$\maclS _s (\rr d)$ to $\maclS _s '(\rr d)$.

\par

\section{More general Orlicz modulation spaces} \label{sec2}

\par

In this section we analyse more general Orlicz modulation spaces,
parameterized with more quasi-Young functions, compared to
what is introduced in Section \ref{sec1}. We prove that if two consecutive
quasi-Youngs are the same, then the Orlicz modulation space remains
the same if one of these parameterizing quasi-Young functions are
removed. In particular it follows
$M^{\Phi ,\Phi}_{(\omega )} = M^{\Phi}_{(\omega )}$ for the Orlicz modulation
spaces considered in Section \ref{sec1}.

\par

\begin{defn}
Let $\mu_j$ be (Borel) measure on
$\rr  {d_j}$, $\mu= \mu_1 \otimes \cdots \otimes \mu _N$,
$\Phi _j$ be quasi-Young functions, $j=1,\dots ,N$, $\omega$
be a weight function and
$f$ be measurable on $\mathbf{R} ^{d_1+ \cdots + d_N}$. Then
$\nm {f}{L^{\Phi _1, \dots, \Phi _N}_{(\omega )}(\mu )}
= \nm{f_{{n-1},\omega}}{L^{\Phi _N}(\mu )}$ where $f_{k,\omega}$, $k=1,\cdots,N-1$
are inductively defined by
\begin{align*}
  f_{1,\omega}(x_2, \dots x_N) &= \nm {f(\cdot,x_2,\dots, x_N)
   \omega(\cdot,x_2,\dots, x_N)} {L^{\Phi _1}(\mu _1)}
\\[1ex]
  f_{k+1,\omega}(x_{k+2}, \dots, x_N) &
  = \nm {f_{k,\omega}(\cdot, x_{k+2}, \dots ,x_N)}
  {L^{\Phi _{k+1}}(\mu _{k+1})}, \quad  k=1,\dots,N-2.
\end{align*}
The space $L^{\Phi _1, \dots, \Phi _N}_{(\omega  )}(\mu )$ consists of all
measurable functions $f$ on $\rr{d_1+ \cdots + d_N}$ such that
$\nm {f}{L^{\Phi _1, \dots, \Phi _N}_{(\omega )}(\mu )}$ is finite, and
the topology of $L^{\Phi _1, \dots, \Phi _N}_{(\omega )}(\mu )$ is
induced by the quasi-norm $\nm {\cdo} {L^{\Phi _1, \dots, \Phi _N}_{(\omega )}(\mu )}$.
\end{defn}

\par

Let
$$
I_{d,N} = \sets {(d_1,\dots ,d_N)\in \zz N_+}{d_1+\cdots +d_N=d}.
$$
For $\dbar =(d_1,\dots ,d_N)\in I_{d,N}$, let
$$
L^{\Phi _1, \dots, \Phi _N}
_{\dbar ,(\omega  )}(\rr {d})
=
L^{\Phi _1, \dots, \Phi _N}_{(\omega  )}(\mu ),
$$
with $\mu =dx _1 \otimes \cdots \otimes dx_N$
with $x_j\in \rr {d_j}$.

\par

If $\Lambda _j\subseteq \rr {d_j}$ are
lattices and $\mu _j$ is the standard
discrete measure on $\Lambda _j$, then we
set
$$
\ell ^{\Phi _1,\dots ,\Phi _N}_{(\omega)}
=
\ell ^{\Phi _1,\dots ,\Phi _N}_{(\omega)}(\Lambda )
\equiv
L^{\Phi _1, \dots, \Phi _N}_{(\omega  )}(\mu ),
\quad
\Lambda =\Lambda _1\times \cdots \times
\Lambda _N,
$$
as usual.

\medspace

When discussing modulation spaces, it is
suitable that $\dbar$ should belong to
$I_{2d,N}^0$, which consists of
all $(d_1,\dots ,d_N) \in I_{2d,N}$ such
that
\begin{equation}\label{Eq:HalfdCond}
d_1+\cdots +d_k=d
\end{equation}
for some $k\in \{ 1,\dots ,N-1\}$, when
$N\ge 2$. We observe
that \eqref{Eq:HalfdCond} implies
$$
d_{k+1}+\cdots +d_N=d.
$$
We observe that $I_{2d,1}=\{ 2d\}$, and
for convenience, we put $I_{2d,1}^0=\{ 2d\}$.

\par

Now suppose that $\dbar \in I_{2d,N}^0$,
$\Lambda _j =\ep \zz {d_j}$, $k$ is
chosen such that \eqref{Eq:HalfdCond} holds,
and let
$$
\Lambda =\Lambda _1\times \cdots \times \Lambda _k
=
\Lambda _{k+1}\times \cdots \times \Lambda _N
= \ep \zz d.
$$
Then we write $\Lambda ^2=\Lambda \times \Lambda$
and
$$
\ell ^{\Phi _1,\dots ,\Phi _N}_{\dbar ,(\omega)}
(\Lambda ^2)
=
\ell ^{\Phi _1,\dots ,\Phi _N}_{\dbar ,(\omega)}
(\ep \zz {2d})
=
\ell ^{\Phi _1,\dots ,\Phi _N}_{(\omega)}
(\Lambda _1\times \cdots \times \Lambda _N).
$$

\par

Let $\Phi _j$ be quasi-Young functions,
$j=1,\dots ,N$,
$\omega \in \mascP _E(\rr {2d})$, $\dbar
\in I_{2d,N}$ 
and $\phi \in \Sigma _1(\rr d)\setminus 0$. Then the Orlicz modulation
space
$$
M^{\Phi _1,\dots ,\Phi _N}_{\dbar ,(\omega )}(\rr d)
$$
consists of all $f\in \Sigma _1'(\rr d)$ such that
$$
\nm f{M^{\Phi _1,\dots ,\Phi _N}_{\dbar ,(\omega )}}
\equiv
\nm {V_\phi f}{L^{\Phi _1,\dots ,\Phi _N}_{\dbar ,(\omega )}}
$$
is finite. By similar arguments as in \cite{ToUsNaOz} it follows that
$M^{\Phi _1,\dots ,\Phi _N}_{\dbar ,(\omega )}(\rr d)$ is a
quasi-Banach space
with quasi-norm
$\nm \cdo{M^{\Phi _1,\dots ,\Phi _N}_{\dbar ,(\omega )}}$,
which is a Banach space and norm, respectively, when $\Phi _j$
is a Young function for every $j\in \{1,\dots ,N\}$.

\par

A common situation is when $\dbar =
(d_0,\dots ,d_0)$ for some integer
$d_0\ge 1$, and then we put
$$
M^{\Phi _1,\dots ,\Phi _N}_{(\omega )}
=
M^{\Phi _1,,\dots ,\Phi _N}_{\dbar ,(\omega )}.
$$

\par

\begin{rem}\label{Rem:GaborExt}
For future references we observe that Proposition
\ref{Prop:Gabor} carry over to Orlicz modulation spaces of the form
$M^{\Phi _1,\dots ,\Phi _N}_{\dbar ,(\omega )}(\rr d)$ when $\Phi _j$ are
quasi-Young functions, $j=1,\dots ,N$,
$\omega \in \mascP _E(\rr {2d})$ and
$\dbar =(d_1,\dots ,d_N)\in I_{2d,N}^0$. In particular
it follows that \eqref{Eq:EquivModNormDisc} takes the form
\begin{multline}\tag*{(\ref{Eq:EquivModNormDisc})$'$}
\nm
{(V_\phi f)(k,\kappa )\}_{k,\kappa \in \ep \Lambda}}
{\ell_{\dbar ,(\omega)}^{\Phi _1, \dots ,\Phi _N}}
\asymp
\nm { (V_\psi f)(k,\kappa)\}_{k,\kappa \in \ep \Lambda}}
{\ell_{\dbar ,(\omega)}^{\Phi _1, \dots ,\Phi _N}}
\\[1ex]
\asymp \Vert f \Vert_{M^{\Phi _1, \dots ,\Phi _N}_{\dbar ,(\omega)}}.
\end{multline}
\end{rem}

\par

\begin{prop}\label{Propn:Cond}
Let $N$, $j_0$ and $d_1,\dots ,d_N$ be positive integers
such that $1\le j_0\le N-1$, $\Lambda _j$ be lattices
in $\rr {d_j}$, $j=1,\dots ,N$,
$$
\Lambda \equiv \Lambda _1\times \cdots \times \Lambda _N
\ni
(n_1,\dots ,n_N)\mapsto a(n_1,\dots ,n_N)\in \mathbf C
$$
and let
$$
b(m_1,\dots ,m_{N-1}) = a(n_1,\dots ,n_d),\qquad n_j\in \Lambda _j,
\ j=1,\dots ,N
$$
where
$$
m_k=
  \begin{cases}
  n_k\in \Lambda _k, & k< j_0,
  \\[1ex]
  (n_{j_0}, n_{j_0+1})\in \Lambda _{j_0}\times \Lambda _{j_0+1}, & k=j_0,
  \\[1ex]
  n_{k+1}\in \Lambda _{k+1}, & k> j_0.
  \end{cases}
$$
Also let $\omega$ be a weight on $\rr {2d}$ and
$\Phi _j,\Psi _k$, $j=1,\cdots N$, $k=1,\dots N-1$,
be quasi-Young functions such that
%
\begin{equation}\label{Eq:YoungFunctionRed}
\Phi _j =
\begin{cases}
\Psi _j, & j\le j_0,
\\[1ex]
\Phi _{j_0}, & j=j_0+1,
\\[1ex]
\Psi _{j-1}, & j>j_0+1.
\end{cases}
\end{equation}
Then
$$
\nm {a} {\ell ^{\Phi _1,\dots ,\Phi _N}_{(\omega)}}
\asymp
\nm {b} {\ell ^{\Psi _1,\dots ,\Psi _{N-1}}_{(\omega)}}
$$
and
$$
\ell ^{\Psi _1,\dots ,\Psi _{N-1}}_{(\omega )}(\Lambda )
=\ell ^{\Phi _1,\dots ,\Phi _N}_{(\omega )}(\Lambda ).
$$
\end{prop}

\par

For the proof we recall that for the
sequence $a$ on
$\zz {d_1+d_2}$ it holds
\begin{equation} \label{Eq:seq1.}
a \in \ell ^\Phi (\zz {d_1+d_2})
\quad \Leftrightarrow  \quad
\sum \limits _{j_1,j_2} \Phi (c \cdot a(j_1, j_2)) <\infty
\end{equation}
for some $c >0$. This implies that
\begin{equation}\label{Eq:seq2.}
 a \in \ell ^{\Phi ,\Phi}(\zz {d_1+d_2})
 \, \Leftrightarrow  \,
 \sum \limits _{j_2} \Phi (c_1 \sum \limits _{j_1}
 \Phi (c_2 (j_2)a(j_1,j_2)))< \infty
\end{equation}
for some $c_1>0$ and a positive sequence $c_2$ on $\zz {d_2}$.

\par

\begin{proof}
We only prove the result in the case $N=2$ and for
$\Lambda _j=\zz {d_j}$. The general case follows by
these arguments and induction, and is left for the reader.

\par

Let $r_0 \in (0,1]$ be chosen such that $\Phi _{0,j}(t)=\Phi _j(t^{\frac{1}{r_0}})$
is a Young function. Then
$$
\nm {a} {\ell ^{\Phi _1, \Phi _2}_{(\omega )}} \asymp
\left(
\nm {|a \cdot \omega|^{r_0}} {\ell ^{\Phi _{0,1}, \Phi _{0,2}}}
\right)^{1/r_0}.
$$
Furthermore, $\nm {|a|} {\ell ^{\Phi _1 \Phi _2}} = \nm {a} {\ell ^{\Phi _1 \Phi _2}}$.
This reduce the result to the case when $\Phi$ is a Young function,
$\omega=1$ and $a \ge 0$.

\par

The result is obviously true when $\Phi = 0$ near origin. In fact for such $\Phi$,
\begin{equation*}
  \ell ^\Phi (\zz {d_1+d_2}) = \ell ^\infty (\zz {d_1+d_2})
  =\ell ^{\infty , \infty} (\zz {d_1+d_2})
  =\ell ^{\Phi ,\Phi} (\zz {d_1+d_2})
\end{equation*}
in view of Proposition \ref{Prop:OrliczModInvariance}.

\par

It remains to consider the case when $\Phi (t)>0$ when $t>0$. Since $\ell ^\Phi$ and
$\ell ^{\Phi ,\Phi}$ do not change when $\Phi (t)$ is replaced by an increasing convex
function which is equal to $c \cdot \Phi (t)$ near $t=0$, where
$c>0$ is a constant, it follows from Proposition \ref{Prop:OrliczModInvariance}
that we may assume that $\Phi(t) \le t$ and that $\Phi$ is increasing.

\par

This gives
\begin{equation*}
  \sum \limits _{j_2} \Phi \left (
  c_1 \sum \limits _{j_1} \Phi (c_2 a(j_1,j_2))
  \right )
  \le c_1 \sum \limits _{j_1, j_2} \Phi
  (c_2 a(j_1,j_2))
\end{equation*}
when $c_1,c_2>0$ are constants.

\par

Hence if
$$
\sum \limits _{j_1, j_2} \Phi
  (c \cdot a(j_1,j_2))<\infty
$$
for some constant $c>0$, then
$$
\sum \limits _{j_2}\Phi
\left( c_1
\sum \limits _{j_1}\Phi (c_2 \cdot a(j_1,j_2))
\right) <\infty
$$
for some constants $c_1,c_2>0$.
By \eqref{Eq:seq1.} and \eqref{Eq:seq2.} we get
\begin{equation}\label{Eq:discinc}
  \ell^\Phi (\zz {d_1+d_2}) \hookrightarrow
   \ell^{\Phi ,\Phi} (\zz {d_1+d_2}).
\end{equation}
We need to deduce the reversed inclusion in \eqref{Eq:discinc}.

\par

First we assume that $a$ has finite support, i.{\,}e.
$a(j_1,j_2)\neq 0$ for at most finite numbers of $(j_1,j_2)$.
Since $\Phi (t) >0$ when $t>0$ it follows that the complementary
Young function $\Phi ^*$ to $\Phi$ fulfills the same properties.

\par

By Proposition 1.20 in \cite{Leo}, we have
\begin{align*}
    \nm {a} {\ell ^\Phi} & \asymp \sup \limits
     _{\nm {b}{\ell^{\Phi ^*}}\le 1} |(a,b)_{\ell^2}|
    \intertext{and}
    \nm {a} {\ell^{\Phi ,\Phi}} & \asymp \sup \limits
    _{\nm {b}{\ell^{\Phi^* ,\Phi^*}}\le 1} |(a,b)_{\ell^2}|.
  \end{align*}
  By a combination of these relations and \eqref{Eq:discinc} we get
  \begin{equation*}
    \nm {a} {\ell^\Phi}  \asymp \sup \limits _{\nm {b}{\ell^{\Phi^*}}\le 1} |(a,b)_{\ell^2}|
    \lesssim \sup \limits _{\nm {b}{\ell^{\Phi^* ,\Phi^*}}\le 1} |(a,b)_{\ell^2}|
    \asymp \nm {a} {\ell^{\Phi ,\Phi}},
  \end{equation*}
and the searched estimate follows for sequences with finite support.

\par

For general $a\ge 0$, let $a_j$, $j\ge 1$ be sequences such that
\begin{equation}\label{Eq:SeqBeppoLevi}
a_j\le a_{j+1}
\quad \text{and}\quad
\lim _{j\to \infty} a_j =a.
\end{equation}
Then Beppo-Levi's theorem gives
$$
\nm {a} {\ell ^\Phi} = \lim _{j\to \infty} \nm {a_j} {\ell ^\Phi}
\lesssim
\lim _{j\to \infty} \nm {a_j} {\ell ^{\Phi ,\Phi}}
=
\nm {a} {\ell ^{\Phi ,\Phi}}.
$$
For general $a$, we may split up $a$ into positive and negative real and
imaginary parts and use \eqref{Eq:SeqBeppoLevi} to get
$$
\nm a{\ell ^\Phi} \lesssim \nm {a} {\ell ^{\Phi ,\Phi}}.
$$
This implies $\ell ^{\Phi ,\Phi} (\zz {d_1+d_2}) \hookrightarrow
\ell ^\Phi (\zz {d_1+d_2})$ and the result follows.
\end{proof}

\par

By combining Propositions \ref{Prop:Gabor}, \ref{Propn:Cond} and
Remark \ref{Rem:GaborExt} we get the following. The details are left for
the reader.

\par



\begin{thm}
Let $N$ and $j_0$ be positive integers
such that $1\le j_0\le N-1$, $\Phi _j$ and $\Psi _k$,
$j=1,\dots ,N$, $k=1,\dots ,N-1$,
be quasi-Young functions such that \eqref{Eq:YoungFunctionRed}
holds and let $\omega \in \mascP_E(\rr {2d})$. Also let
$$
\dbar = (d_1,\dots ,d_N)\in I_{2d,N}^0
\quad \text{and}\quad
\dbar _0= (d_{0,1},\dots ,d_{0,N-1})\in I_{2d,N-1}^0
$$
be such that
$$
d_{0,j} =
\begin{cases}
d_j, & j<j_0,
\\[1ex]
d_{j_0}+d_{j_0+1}, & j=j_0,
\\[1ex]
d_{j+1}, & j>j_0.
\end{cases}
$$
Then
  \begin{align*}
    M^{\Phi _1,\dots ,\Phi _N}_{\dbar, (\omega )}(\rr {d})
    &=
    M^{\Psi _1,\dots ,\Psi _{N-1}} _{\dbar _0,(\omega )}(\rr d)
    \intertext{and}
    \nm {f} {M^{\Phi _1,\dots ,\Phi _N}_{\dbar, (\omega )}}
    &\asymp
    \nm {f} { M^{\Psi _1,\dots ,\Psi _{N-1}} _{\dbar _0,(\omega )}}
  \end{align*}
  when $f \in \Sigma '(\rr d)$.
\end{thm}

\par

\begin{cor}
Let $N$ be a positive integer, $\dbar \in I_{2d,N}^0$,
  $\Phi$ be a quasi-Young function and $\omega \in
  \mascP_E(\rr {2d})$. Then
  \begin{align*}
    M^{\Phi ,\dots ,\Phi}_{\dbar, (\omega )}(\rr d) &=
    M^\Phi _{(\omega )}(\rr d)
    \intertext{and}
    \nm {f} { M^{\Phi ,\dots ,\Phi}_{\dbar, (\omega )}}
    &\asymp
    \nm {f} { M^\Phi _{(\omega )}}
  \end{align*}
when $f \in \Sigma '(\rr d)$.
\end{cor}

\par

\section{Continuity of pseudo-differential operators on Orlicz
modulation spaces}\label{sec3}

\par

In this section we deduce continuity properties of pseudo-differential operators
when acting on Orlicz modulation spaces. The main results are
Theorems \ref{Thm:PseudoCont} and \ref{Thm:PseudoCont2} which deal
with such operators with symbols in
with symbols in $M^{\infty ,r_0}_{(\omega )}(\rr {2d})$ and
$M^{\Phi ,\Phi}_{(\omega )}(\rr {2d})$, respectively, where $r_0\in (0,1]$
and $\Phi$ is a quasi-Young functions.

\par

In the first part we deduce related continuity
results for suitable matrix operators. In the second part we combine these results and
Gabor analysis results from the previous section to establish the continuity results
for the pseudo-differential operators.

\par

In the following definition we recall some matrix classed, considered in
\cite{Toft16}. Here we observe that we may identify $\Lambda \times \Lambda$
matrices with sequences on $\Lambda \times \Lambda$, when $\Lambda$
is a lattice in $\rr d$.

\par

\begin{defn}\label{Def:MatrixClasses}
Let $p,q \in (0,\infty ]$, $\Phi _1,\Phi _2$ be quasi-Young functions,
$\omega \in \mascP _E(\rr {2d})$, $\Lambda$ be lattice in $\rr d$ and
let $T$ be the map on $\ell _0'(\Lambda \times \Lambda )$, given by
$$
(Ta)(j,k) = a(j,j-k),\qquad a\in \ell _0'(\Lambda \times \Lambda ),\ j,k\in \Lambda .
$$
\begin{enumerate}
\item The set  $\mathbb{U}_0' (\Lambda \times \Lambda )$ consists of all (formal)
matrices
\begin{equation}\label{Eq:MatrixClasses}
A=(a(j,k))_{j,k\in \Lambda}
\end{equation}
with entries $a(j,k)$ in $\mathbf C$, and
$\mathbb{U}_0 (\Lambda \times \Lambda )$ consists of all
$A$ in \eqref{Eq:MatrixClasses}
such that at most finite numbers of $a(j,k)$ are nonzero.

\vrum

\item The set $\mathbb{U}^{p,q} _{(\omega )}(\Lambda \times \Lambda )$ consists of
all matrices $A=(a(j,k))_{j,k\in \Lambda}$ such that
$$
\nm A{\mathbb{U}^{p,q}_{(\omega )}}
\equiv \nm {T(a\cdot \omega )}{\ell ^{p,q}},
$$
is  finite.

\vrum

\item The set
$\mathbb{U}^{\Phi _1,\Phi _2} _{(\omega )}(\Lambda \times \Lambda )$
consists of all matrices $A=(a(j,k))_{j,k\in \Lambda}$ such that
$$
\nm A{\mathbb{U}^{\Phi _1,\Phi _2}_{(\omega )}}
\equiv \nm {T(a\cdot \omega )}{\ell ^{\Phi _1,\Phi _2}},
$$
is finite.
\end{enumerate}
\end{defn}

\par

\begin{rem}\label{Rem:MatrixClasses}
Let $p\in (0,\infty ]$, $\Phi$ be a quasi-Young function
and $\omega \in \mascP _E(\rr {2d})$. Then it follows from
Proposition \ref{Propn:Cond} and straight-forward changes of
variables that the following is true. The details are left for the
reader.
\begin{enumerate}
\item If $A_0=(a(j,k))_{j,k\in \zz d}$ is a matrix, then
$A_0\in \mathbb U ^{p,p}_{(\omega )}(\zz {2d})$,
if and only if
$a\in \ell ^{p,p}_{(\omega )}(\zz {2d}) = \ell ^p_{(\omega )}(\zz {2d})$,
and
$$
\nm {A_0}{\mathbb U ^{p,p}_{(\omega )}}
=
\nm a{\ell ^{p,p}_{(\omega )}}
=
\nm a{\ell ^{p,p}_{(\omega )}}.
$$

\vrum

\item  If $A_0=(a(j,k))_{j,k\in \zz d}$ is a matrix, then
$A_0\in \mathbb U ^{\Phi ,\Phi}_{(\omega )}(\zz {2d})$,
if and only if
$a\in \ell ^{\Phi ,\Phi}_{(\omega )}(\zz {2d})
= \ell ^\Phi _{(\omega )}(\zz {2d})$,
and
$$
\nm {A_0}{\mathbb U ^{\Phi ,\Phi}_{(\omega )}}
=
\nm a{\ell ^{\Phi ,\Phi}_{(\omega )}}
=
\nm a{\ell ^{\Phi ,\Phi}_{(\omega )}}.
$$
\end{enumerate}
\end{rem}

\par

Next we discuss continuity for certain matrix operator when acting on discrete
Orlicz spaces. We recall that if $\Lambda \subseteq \rr d$
is a lattice, $\omega_1, \omega_2 \in \mascP_E (\rdd)$ and
$\omega \in \mascP_E (\rr {4d})$ are such that
\begin{equation}\label{Eq:WeightOpCondDiscrete}
\frac{\omega _2(j)}{\omega _1(k)}
\le \omega (j,k),\qquad j,k\in \Lambda ^2,
\end{equation}
$r_0\in (0,1]$ and $p,q\in [r_0,\infty ]$, then
\cite[Theorem 2.3]{Toft16} shows that $A_0$
from $\ell _0(\Lambda ^2)$ to $\ell _0'(\Lambda ^2)$ is uniquely extendable
to a continuous map from $\ell _{(\omega _1)}^{p,q}(\Lambda ^2)$ to
$\ell _{(\omega _2)}^{p,q}(\Lambda ^2)$. The following result extends this result
to discrete Orlicz spaces.

\par

\begin{thm}\label{Thm:Aop}
Let $\ep >0$, $N\ge 1$ be an integer,
$\dbar \in I_{2d,N}^0$, $\Phi _1,\dots ,\Phi _N$
be quasi Young functions of order $r_0 \in(0,1]$,
$\omega_1, \omega_2 \in \mascP_E (\rdd )$,
$\omega \in \mascP_E
(\rr {4d})$ be such that 
\eqref{Eq:WeightOpCondDiscrete} holds.
If $A\in \mathbb{U}^{\infty ,r_0}
_{(\omega )}(\ep \zz {4d})$,
then $A$ from
$\ell _{(\omega _1)}^\infty
(\ep \zz {2d})$ to $\ell _{(\omega _2)}^\infty
(\ep \zz {2d})$ restricts to a
continuous map from
$\ell^{\Phi _1,\dots ,\Phi _N}
_{\dbar ,(\omega_1)}(\ep \zz {2d})$
to $\ell^{\Phi _1,\dots ,\Phi 
_N}_{\dbar ,(\omega_2)}(\ep \zz {2d})$ and
\begin{equation}\label{op1}
\nm {Af}{\ell ^{\Phi _1,\dots ,\Phi _N}
_{\dbar ,(\omega _2)}}
\leq
\nm A{\mathbb{U}^{\infty ,r_0}_{(\omega )}}
\nm f{\ell^{\Phi _1,\dots ,\Phi _N}
_{\dbar ,(\omega _1)}},
\quad
f\in \ell^{\Phi _1,\dots ,\Phi _N}
_{\dbar ,(\omega _1)}(\ep \zz {2d}).
\end{equation}
\end{thm}

\par


\par

We need the following lemma for the proof of Theorem
\ref{Thm:Aop}. We omit the proof since the result is
a consequence of \cite[Lemma 3.1]{ToUsNaOz}.

\par

\begin{lemma}\label{Lemma:QuasBanachSeqConv}
Let $\Lambda \subseteq \rr d$
be a lattice, $\mascB \subseteq \ell _0'(\Lambda )$ be a quasi-Banach
of order $r_0\in (0,1]$, with quasi-norm $\nm \cdo {\mascB}$.
If
$$
\nm {f(\cdo -j)}{\mascB} = \nm f{\mascB},
\qquad f\in \mascB ,\ j\in \Lambda ,
$$
then the discrete convolution map $(f,g)\mapsto f*_\Lambda g$ from
$\ell ^{r_0}(\Lambda )\times \ell ^{r_0}(\Lambda )$ to $\ell ^{r_0}(\Lambda )$
extends uniquely to a continuous map from
$\mascB \times \ell ^{r_0}(\Lambda )$ to
$\mascB$, and
$$
\nm {f*g}{\mascB} \le \nm f{\mascB}\nm g{\ell ^{r_0} (\Lambda )},
\qquad f\in \mascB ,\ g\in \ell ^{r_0}(\Lambda ).
$$
\end{lemma}

\par
%
%
%

\par

\begin{proof}[Proof of Theorem \ref{Thm:Aop}]
We only prove the result in the case $N=2$.
For general $N$, the result follows by similar
arguments, and is left for the reader.
Let $f\in \ell^{\Phi _1,\Phi _2}_{(\omega_1)}
(\ep \zz {2d})$ and set $g=Af$.

\par

First we consider the case
when $A\in \mathbb{U}_0 (\ep \zz {4d})$
and let
\begin{equation*}
a_{\omega}(j,k)=|a(j,k)\omega (j,k)|,\quad
f_{\omega_1}(k)=|f(k)\omega_1(k)|
\end{equation*}
and
$$
g_{\omega_2}(j)=|g(j)\omega_2(j)|.
$$
We get
\begin{multline*}
g_{\omega _2}(j)=|Af(j)\omega _2(j)|
\\[1ex]
\leq
\sum \limits_{k\in \ep \zz {2d}}
|a(j,k)|f(k)\omega_1(k)\omega (j,k)|
\\[1ex]
= \sum \limits _{k\in \ep \zz {2d}}
|a_{\omega}(j,j-k)f_{\omega_1}(j-k)|
\\[1ex]
\le
\sum \limits _{k\in \ep \zz {2d}} 
h_{\omega}(k)f_{\omega_1}(j-k)
=(h_{\omega}*f_{\omega _1})(j),
\end{multline*}
where $h_{\omega}(k)=\sup \limits
_{j\in \ep \zz {2d}} a_{\omega}(j,j-k)$.

\par

By Lemma \ref{Lemma:QuasBanachSeqConv} we get
\begin{multline*}
\Vert Af\Vert_{\ell^{\Phi _1,\Phi _2}_{(\omega_2)}} =
\Vert g_{\omega_2}\Vert_{\ell^{\Phi _1,\Phi _2}}
\leq
\Vert h_{\omega}*f_{\omega_1}\Vert_{\ell^{\Phi _1,\Phi _2}}
\leq \Vert h_{\omega}\Vert_{\ell^{r_0}}
\Vert f_{\omega _1}\Vert_{\ell^{\Phi _1,\Phi _2}}
\\[1ex]
= \Vert A\Vert_{\mathbb{U}^{\infty ,r_0}_{(\omega )}}
\Vert f\Vert_{\ell^{\Phi _1,\Phi _2}_{(\omega_1)}},
\end{multline*}
and the result follows in this case.

\par

For general $A\in \mathbb{U}^{\infty ,r_0}
_{(\omega )}(\ep \zz {4d})$
we decompose $A$ and $f$ into
\begin{equation}\label{Eq:DecompRealImagPosNeg}
A=A_1 -A_2 +i(A_3 -A_4 )\quad \text{and}
\quad f=f_1 -f_2 +i(f_3 -f_4 ),
\end{equation}
where $A_j$ and $f_k$ only have non-negative
entries,
chosen as small as possible.
By Beppo-Levi's theorem and
the estimates above it follows
that $A_j f_k$ is uniquely defined
as an element in
$\ell^{\Phi _1\Phi _2}_{(\omega_2)}(\ep \zz {2d})$.
It also
follows from these estimates that \eqref{op1} holds.
\end{proof}

\par

\begin{rem}\label{Rem:Aop}
Let
$$
A=(a(j,k))_{j,k\in \ep \zz {2d}}\in
\mathbb U_0'( \ep \zz {4d})
\quad \text{and}\quad
f = \{ f(j)\} _{j\in \ep \zz {2d}}
\in \ell _0'(\ep \zz {2d}).
$$
Then $A_n$ in
\eqref{Eq:DecompRealImagPosNeg} are given by
$$
A_n=(a_n(j,k))_{j,k\in \ep \zz {2d}}
$$
where
\begin{alignat*}{2}
a_1(j,k) &= \max (\operatorname{Re}(a(j,k)),0),
\quad
a_2(j,k) &= \min (\operatorname{Re}(a(j,k)),0),
\\[1ex]
a_3(j,k) &= \max (\operatorname{Im}(a(j,k)),0),
\quad
a_4(j,k) &= \min (\operatorname{Im}(a(j,k)),0),
\end{alignat*}
and $f_n = \{ f_n(j)\} _{j\in \ep \zz {2d}}$,
are obtained in the same way after
each $a_n(j,k)$ and $a(j,k)$ are replaced by
$f_n(j)$ and $f(j)$, respectively.
\end{rem}

\par

Before we discuss continuity properties of pseudo-differential operators
on Orlicz modulation spaces, we have the following result concerning
operator classes
$$
\sets {\op _A(a)}{a\in M^{\Phi _1,\Phi _2}_{(\omega )}(\rr {2d})}
$$
of continuous operators from $\Sigma _1(\rr d)$ to $\Sigma _1'(\rr d)$.
Here recall \cite[Proposition 1.9]{Toft19} for analogous relations
for pseudo-differential operators
with symbols in (ordinary) modulation spaces.

\par

\begin{prop}\label{Prop:CalculiTransferOrlModSp}
Let $N\ge 1$ be an integer, $\dbar
\in I_{4d,N}^0$, $A\in \GL (d,\mathbf R)$,
$\Phi _1,\dots ,\Phi _N$ be quasi-Young
functions, $\omega \in \mascP _E(\rr {4d})$
and let
$$
\omega _A(x,\xi ,\eta ,y) =
\omega (x-Ay,\xi -A^*\eta ,\eta ,y).
$$
Then the following is true:
\begin{enumerate}
\item the map $e^{i\scal {AD_\xi}{D_x}}$ from
$M^\infty _{(\omega )}(\rr {2d})$
to $M^\infty _{(\omega _A)}(\rr {2d})$ restricts
to a homeomorphism from
$M^{\Phi _1,\dots ,\Phi _N}
_{\dbar , (\omega )}(\rr {2d})$
to $M^{\Phi _1,\dots ,\Phi _N}
_{\dbar ,(\omega _A)}(\rr {2d})$;

\vrum

\item the set
$$
\sets {\op _A(a)}
{a\in M^{\Phi _1,\dots ,\Phi _N}
_{\dbar ,(\omega _A)}(\rr {2d})}
$$
of operators from $\Sigma _1(\rr d)$ to
$\Sigma _1'(\rr d)$ is independent of
$A\in \GL (d,\mathbf R)$.
\end{enumerate}
\end{prop}

\par

\begin{proof}
We only prove the result in the case $N=2$.
For general $N$, the result follows by similar
arguments, and is left for the reader.

\par

It suffices to prove (1) in view of Proposition 
\ref{Prop:CalculiTransfer}.

\par

Let $a\in M^{\Phi _1,\Phi _2} _{(\omega )}(\rr {2d})$,
$\phi \in \Sigma _1(\rr d)$, $\psi = e^{i\scal {AD_\xi}{D_x}}\phi$
and $b=e^{i\scal {AD_\xi}{D_x}}a$. Then it follows from
Theorem 3.1 and (3.1) in \cite{AbCaTo} that $\psi \in \Sigma _1(\rr d)$ and
$$
|V_\psi b(x,\xi ,\eta ,y)\omega _A(x,\xi ,\eta ,y)|
=
|V_\phi a(x-Ay,\xi -A^*\eta ,\eta ,y)\omega (x-Ay,\xi -A^*\eta ,\eta ,y)| .
$$
By applying the $L^{\Phi _1}$ quasi-norm with respect to the $(x,\xi )$ variables gives
\begin{multline*}
\nm {V_\psi b(\cdo ,\eta ,y)\omega _A(\cdo ,\eta ,y)}{L^{\Phi _1}}
\\[1ex]
=
\nm {V_\phi a(\cdo -(Ay,A^*\eta ),\eta ,y)\omega (\cdo -(Ay,A^*\eta ),\eta ,y)}{L^{\Phi _1}}
\\[1ex]
=
\nm {V_\phi a(\cdo ,\eta ,y)\omega (\cdo ,\eta ,y)}{L^{\Phi _1}},
\end{multline*}
and applying the $L^{\Phi _2}$ quasi-norm with respect to the $(y,\eta )$ on
the last equality gives
$$
\nm {V_\psi b\cdot \omega _A}{L^{\Phi _1,\Phi _2}}
=
\nm {V_\phi a\cdot \omega}{L^{\Phi _1,\Phi _2}}.
$$
This gives
$$
\nm b{M^{\Phi _1,\Phi _2}_{(\omega _A)}} = \nm a{M^{\Phi _1,\Phi _2}_{(\omega )}},
$$
and the result follows.
\end{proof}

\par

We have now the following continuity result for 
pseudo-differential operators acting on Orlicz 
modulation spaces. Here and in what follows, $A^*$ 
denotes the transpose of the matrix $A$,
and the involved weight functions should
satisfy
\begin{equation}\label{Eq:PseudoContWeightCond}
\frac {\omega _2(x,\xi )}{\omega _1(y,\eta )}
\lesssim
\omega (x+A(y-x),
\eta +A^*(\xi -\eta ), \xi -\eta ,y-x),
\end{equation}

\par

\begin{thm}\label{Thm:PseudoCont}
Let $A\in M(d,\mathbf{R})$, $\Phi _1,\Phi _2$ be quasi Young functions of order
$r_0 \in(0,1]$, $\omega \in \mascP_E (\rr {4d})$ and $\omega_1,\omega_2
\in \mascP _E(\rr {2d})$ be such that
\eqref{Eq:PseudoContWeightCond} holds,
and let $a\in M^{\infty ,r_0}_{(\omega )}(\rr {2d})$. Then
$\Op_A(a)$ from $\Sigma _1(\rr d)$ to $\Sigma _1'(\rr d)$
is uniquely extendable to a continuous map from
$M^{\Phi _1,\Phi _2}_{(\omega_1)}(\rd )$ to
$M^{\Phi _1,\Phi _2}_{(\omega_2)}(\rd )$, and
\begin{equation}\label{Eq:PseudoCont}
\nm {\Op_A(a)}{M^{\Phi _1,\Phi _2}_{(\omega_1)}(\rd )\to M^{\Phi _1,\Phi _2}_{(\omega_2)}(\rd )}
\lesssim
\nm a{M^{\infty ,r_0}_{(\omega )}}.
\end{equation}
\end{thm}

\par

The previous result can be generalized as in the following.

\par

\begin{thm}\label{Thm:PseudoContExt}
Let $A\in M(d,\mathbf{R})$, 
$N$ be a positive integer, $\dbar \in I_{d,N}$, $\Phi _j$,
$j=1,\dots ,N$, be quasi-Young functions
of order
$r_0 \in(0,1]$, $\omega \in \mascP_E (\rr {4d})$ and $\omega_1,\omega_2
\in \mascP _E(\rr {2d})$ be such that
\eqref{Eq:PseudoContWeightCond} holds,
and let $a\in M^{\infty ,r_0}_{(\omega )}(\rr {2d})$. Then
$\Op_A(a)$ from $\Sigma _1(\rr d)$ to $\Sigma _1'(\rr d)$
is uniquely extendable to a continuous map from
$M^{\Phi _1,\dots ,\Phi _N}_{\dbar, (\omega_1)}(\rd )$ to
$M^{\Phi _1,\dots ,\Phi _N}_{\dbar ,(\omega_2)}(\rd )$, and
\begin{equation}\label{Eq:PseudoContExt}
\nm {\Op_A(a)}{M^{\Phi _1,\dots ,\Phi _N}_{\dbar, (\omega_1)}(\rd )
\to M^{\Phi _1,\dots ,\Phi _N}_{\dbar, (\omega_2)}(\rd )}
\lesssim
\nm a{M^{\infty ,r_0}_{(\omega )}}.
\end{equation}
\end{thm}


\par

We only prove Theorem \ref{Thm:PseudoCont}. Theorem
\ref{Thm:PseudoContExt} follows by similar arguments and
is left for the reader.

\par

We need some preparations for the proof of Theorem \ref{Thm:PseudoCont}.
First we have the following extension of \cite[Lemma 3.3]{Toft16} to
the case of Orlicz modulation spaces. 

\par

\begin{lemma}\label{Lemma:PseudoCont}
Let $\Lambda$, $\phi _1$, $\phi _2$, $\fy$, $\psi$ and $\ep >0$ be as in Lemma
\ref{Lemma:GoodFrames}. Also let $v\in \mascP _E (\rr {4d})$, $\Phi _1$,
$a\in M^\infty _{(1/v)}(\rr {2d})$,
\begin{multline*}
c_0(\mabfj ,\mabfk) \equiv (V_\psi a)(j,\kappa ,\iota -\kappa ,k-j)e^{i\scal {k-j}\kappa},
\\[1ex]
\text{where}
\quad \mabfj =(j,\iota )\in \ep \Lambda ^2 ,\ \mabfk = (k,\kappa )\in \ep \Lambda ^2.
\end{multline*}
and let $A_a$ be the matrix $A_a=(c_0(\mabfj ,\mabfk ))
_{\mabfj ,\mabfk \in \ep \Lambda ^2}$.
Then the following is true:
\begin{enumerate}
\item if $\Phi _1$, $\Phi _2$ are quasi-Young functions and
$\omega ,\omega _0\in \mascP _E(\rr {4d})$ satisfy
\begin{equation}\label{omega0omegaRel}
\omega (x,\xi ,y,\eta )\asymp \omega _0(x,\eta ,\xi -\eta ,y-x),
\end{equation}
then $a\in M^{\Phi _1,\Phi _2}_{(\omega _0)}(\rr {2d})$, if and only if
$A_a\in \mathbb U^{\Phi _1,\Phi _2}_{(\omega )}
(\ep (\Lambda ^2\times \Lambda ^2) )$, and then
$$
\nm a{M^{\Phi _1,\Phi _2}_{(\omega _0)}}\asymp
\nm {A_a}{\mathbb U^{\Phi _1,\Phi _2}_{(\omega )}(\ep (\Lambda ^2\times \Lambda ^2))}
\text ;
$$

\vrum

\item $\op (a)$ as map from $\Sigma _1(\rr d)$ to $\Sigma _1'(\rr d)$ is given
by
\begin{equation}\label{OpaFactorization}
\op (a) = D_{\phi _1}^{\ep ,\Lambda} \circ A_a \circ C_{\phi _2}^{\ep ,\Lambda}.
\end{equation}
\end{enumerate}
\end{lemma}

\par

\begin{proof}
We have
$$
|c_0(\mabfj ,\mabfj -\mabfk )| = | (V_\Psi a) (j,\iota -\kappa ,\kappa ,-k)|.
$$
Hence, Proposition \ref{Prop:Gabor} gives
$$
\nm {A_a}{\mathbb U^{\Phi _1,\Phi _2}_{(\omega )}(\ep (\Lambda ^2 \times \Lambda ^2 ))}
=
\nm {V_\Psi a}{\ell ^{\Phi _1,\Phi _2}_{(\omega _0)}(\ep (\Lambda ^2 \times \Lambda ^2 ))}
\asymp
\nm a{M^{\Phi _1,\Phi _2}_{(\omega _0)}},
$$
and (1) follows.

\par

The assertion (2) is the same as assertion (2) in
\cite[Lemma 3.3]{Toft16}.
The proof is therefore omitted.
\end{proof}

\par

\begin{proof}[Proof of Theorem \ref{Thm:PseudoCont}]
By Proposition \ref{Prop:CalculiTransferOrlModSp}
we may assume that $A=0$.

\par

Let $a$, $A_a$, $\phi _1$ and $\phi _2$ be the same as in
Proposition \ref{Prop:AnalysisSynthOp} and Lemma
\ref{Lemma:PseudoCont}. Then by
%
Proposition \ref{Prop:AnalysisSynthOp}, Theorem \ref{Thm:Aop}
and Lemma \ref{Lemma:PseudoCont} we get
$$
\nm {\Op(a)}{M^{\Phi _1,\Phi _2}_{(\omega_1)}\rightarrow M^{\Phi _1,\Phi _2}_{(\omega_2)}}
\lesssim J_1\cdot J_2\cdot J_3,
$$
where
\begin{align}
J_1 &= \nm {D_{\phi _1}}{\ell^{\Phi _1,\Phi _2}_{(\omega _2)}
\rightarrow
M^{\Phi _1,\Phi _2}_{(\omega _2)}}<\infty ,
\label{Eq:ContSynthOpOrMod}
\\[1ex]
J_2 &= \nm {A_a}{\ell^{\Phi _1,\Phi _2}_{(\omega _2)}
\rightarrow
\ell^{\Phi _1,\Phi _2}_{(\omega _2)}}<\infty
\intertext[0ex]{and}
J_3 &= \nm {C_{\phi _2}}{M^{\Phi _1,\Phi _2}_{(\omega _1)}
\rightarrow
M^{\Phi _1,\Phi _2}_{(\omega _1)}}<\infty.
\label{Eq:ContAnalOpOrMod}
\end{align}
This gives the asserted continuity. The uniqueness
follows from the facts that
$$
M^{\Phi _1,\Phi _2}_{(\omega _j)} (\rd )\subseteq
M^{\infty}_{(\omega _j)} (\rd ),
$$
in view of Proposition \ref{Prop:OrliczModInvariance} and that $\Op (a)$ is
uniquely defined as a continuous operator from
$M^{\infty}_{(\omega_1)}(\rd )$ to $M^{\infty}_{(\omega_2)}(\rd )$, in view of
\cite[Theorem 3.1]{Toft16}.
\end{proof}

\par

We have also the following.

\par 

\begin{thm}\label{Thm:PseudoCont2}
Let $A\in M(d,\mathbf{R})$, $\Phi _0$ be a Young function,
$\Phi _0^*$ the complementary Young function of $\Phi _0$,
$\Phi$ be a quasi-Young function such that
\begin{equation}\label{Eq:StrictConvave}
  \lim _{t\to 0+} \frac t{\Phi (t)}
\end{equation}
is finite and let $\omega \in \mascP_E (\rr {4d})$
and $\omega_1,\omega_2 \in \mascP _E(\rr {2d})$
be such that \eqref{Eq:PseudoContWeightCond}
holds. Then the following is true:
\begin{enumerate}
\item if $a \in M^{\Phi _0} _{(\omega )}(\rr {2d})$, then
$\Op _A(a)$ from $M^1_{(\omega _1)}(\rr d)$ to
$M^\infty _{(\omega _2)}(\rr d)$
is extendable to a continuous map from $M^{\Phi _0^* }_{(\omega _1)}(\rd)$
to $M^{\Phi _0}_{(\omega _2)}(\rd)$ and
\begin{equation*}
\nm {\Op _A(a)}
{M^{\Phi _0^* }_{(\omega _1)}(\rd )\to M^{\Phi _0}_{(\omega _2)}(\rd )}
\lesssim
\nm a{M^{\Phi _0}_{(\omega )}(\rr{2d})}\text ;
\end{equation*}

\vrum

\item if $a \in M^{\Phi} _{(\omega )}(\rr {2d})$, then
$\Op _A(a)$ from $M^1_{(\omega _1)}(\rr d)$ to
$M^\infty _{(\omega _2)}(\rr d)$
is uniquely extendable to a continuous map from
$M^{\infty}_{(\omega _1)}(\rd)$
to $M^{\Phi}_{(\omega _2)}(\rd)$, and
\begin{equation*}
\nm {\Op _A(a)}
{M^{\infty}_{(\omega _1)}(\rd )\to M^{\Phi}_{(\omega _2)}(\rd )}
\lesssim
\nm a{M^{\Phi}_{(\omega )}(\rr{2d})}.
\end{equation*}
\end{enumerate}
\end{thm}

\par

\begin{proof}
By Proposition \ref{Prop:CalculiTransferOrlModSp}
we may assume that $A=0$.

\par

Let $\Lambda \subseteq \rr d$ be a lattice,
$A_0=(a(j,k))_{j,k\in \Lambda}
\in \mathbb U ^{\Phi _0,\Phi _0} _{(\omega )}(\Lambda \times \Lambda )$
and $f\in \ell ^{\Phi _0^*}_{(\omega _1)}(\Lambda )$ be such that $a(j,k)\ge 0$ and
$f(j)\ge 0$ for every $j,k\in \Lambda$.
We have
$$
0\le (A_0f)(j)\omega _2(j)
=
(a(j,\cdo ),f)\omega _2(j)
\lesssim
\nm {a(j,\cdo )\omega (j,\cdo )}{\ell ^{\Phi _0}}\nm {f\cdot \omega _1}{\ell ^{\Phi _0^*}}.
$$
By applying the $\ell ^{\Phi _0}$ norm
and using Remark \ref{Rem:MatrixClasses} we get
\begin{equation}\label{Eq:MatrixOpEst4}
\nm {A_0f}{\ell ^{\Phi _0}_{(\omega _2)}}
\lesssim
\nm a{\ell ^{\Phi _0,\Phi _0}_{(\omega )}}\nm f{\ell ^{\Phi _0^*}_{(\omega _1)}}
\asymp
\nm {A_0}{\mathbb U^{\Phi _0,\Phi _0}_{(\omega )}}\nm f{\ell ^{\Phi _0^*}_{(\omega _1)}},
\end{equation}
which implies that $A_0f$ makes sense
as an element in
$\ell ^{\Phi _0}_{(\omega _2)}(\Lambda )$.

\par

For general $A_0=(a(j,k))_{j,k\in \Lambda} \in
\mathbb U ^{\Phi _0,\Phi _0}
_{(\omega )}(\Lambda \times \Lambda )$
and $f\in \ell ^{\Phi _0^*}_{(\omega _1)}
(\Lambda )$, we define $A_0f$ in similar
ways as in the proof of Theorem \ref{Thm:Aop},
by splitting up $A_0$ and $f$
into positive and negative parts of their real
and imaginary parts. By
\eqref{Eq:MatrixOpEst4} we obtain
\begin{equation}\label{Eq:MatrixBoundednessPhi0Phi0*}
\nm {A_0}{\ell ^{\Phi _0^*}_{(\omega _1)}(\Lambda )
\to
\ell ^{\Phi _0}_{(\omega _2)}(\Lambda )}
\lesssim
\nm {A_0}{\mathbb U^{\Phi _0,\Phi _0}_{(\omega )}}.
\end{equation}

\par

Now let $a\in M^{\Phi _0}_{(\omega )}(\rr {2d})$
and $f\in M^{\Phi _0^*}_{(\omega _1)}(\rr d)$.
Then we define $\op (a)f$
by \eqref{OpaFactorization}. The asserted
continuity in (1) now follows
from Proposition \ref{Propn:Cond},
\eqref{Eq:ContSynthOpOrMod},
\eqref{Eq:ContAnalOpOrMod}
and \eqref{Eq:MatrixBoundednessPhi0Phi0*}.

\par

Next let $\Phi$ be as in (2) and let
$A_0=(a(j,k))_{j,k\in \Lambda}
\in \mathbb U ^{\Phi ,\Phi }_{(\omega )}
(\Lambda \times \Lambda )$
and $f\in \ell ^{\infty}_{(\omega _1)}(\Lambda )$
be such that $a(j,k)\ge 0$ and
$f(j)\ge 0$ for every $j,k\in \Lambda$. Then
\begin{multline*}
0\le (A_0f)(j)\omega _2(j)
=
(a(j,\cdo ),f)\omega _2(j)
\\[1ex]
\lesssim
\nm {a(j,\cdo )\omega (j,\cdo )}{\ell ^1}\nm {f\cdot \omega _1}{\ell ^\infty}.
\lesssim
\nm {a(j,\cdo )\omega (j,\cdo )}{\ell ^{\Phi}}
\nm f{\ell ^\infty _{(\omega _1)}},
\end{multline*}
where the last inequality follows from
Proposition \ref{Prop:OrliczModInvariance}
and
\eqref{Eq:StrictConvave}. By applying the
$\ell ^\Phi$ quasi-norm and splitting up
general $A_0=(a(j,k))_{j,k\in \Lambda}
\in \mathbb U ^{\Phi ,\Phi }_{(\omega )}
(\Lambda \times \Lambda )$
and $f\in \ell ^{\infty}_{(\omega _1)}(\Lambda )$
into positive and negative
real and imaginary parts, we obtain
\begin{equation}
\label{Eq:MatrixBoundednessInftyPhi}
\nm {A_0}{\ell ^\infty _{(\omega _1)}(\Lambda )
\to \ell ^{\Phi}_{(\omega _2)}(\Lambda )}
\lesssim
\nm {A_0}{\mathbb U^{\Phi ,\Phi}_{(\omega )}}.
\end{equation}
The asserted continuity in (2) now follows by combining
Proposition \ref{Propn:Cond},
\eqref{Eq:ContSynthOpOrMod},
\eqref{Eq:ContAnalOpOrMod}
and \eqref{Eq:MatrixBoundednessInftyPhi}.

\par

The asserted uniqueness follows from the
fact that if
$a\in M^\Phi _{(\omega )}(\rr {2d})$,
then $a\in M^1_{(\omega )}(\rr {2d})$
in view of Proposition
\ref{Prop:OrliczModInvariance}
and \eqref{Eq:StrictConvave}. Hence, if
$f\in M^\infty _{(\omega _1)}(\rr d)$,
then $\op (a)f$ is uniquely defined as an element in
$M^1_{(\omega _2)}(\rr d)$ (see e.{\,}g. \cite[Theorem 3.1]{Toft16}).
This in turn implies that
$\op (a)f$ is uniquely defined as an element in
$M^\Phi _{(\omega _2)}(\rr d)$, and the result follows.
\end{proof}

\par

\section{Symbol product estimates on Orlicz modulation spaces}
\label{sec4}

\par

In this section we show that if $\omega _j$ are suitable weights, $j=0,1,2$,
$\Phi _1,\Phi _2$ are quasi-Young functions of order $r_0\in (0,1]$,
$a_1\in M^{\Phi _1,\Phi _2}_{(\omega _1)}$ and
$a_2\in M^{\infty ,r_0}_{(\omega _2)}$, then
$\op _A(a_1)\circ \op _A(a_2)$ equals $\op _A(b)$ for some
$a_1\in M^{\Phi _1,\Phi _2}_{(\omega _0)}$.

\par

An essential condition on the weight functions is
\begin{equation} \label{Eq:weight1}
\omega _0(T _A(Z,X))\lesssim \omega _1(T _A(Y,X))
\omega _2(T _A(Z,Y)),\quad X, Y, Z \in(\rdd),
\end{equation}
where
\begin{multline} \label{Eq:weight2}
T _A(X,Y)=(y+A(x-y),\xi+A^{*}(\eta -\xi),\eta-\xi ,x-y),
\\[1ex]
X=(x,\xi) \in\rdd, Y =(y,\eta)\in \rdd.
\end{multline}

\par

\begin{thm}\label{Thm:main}
Let $A\in M(d,\mathbf{R})$ and suppose that $\omega _k \in \mascP
_E(\rr{4d})$, $k=0,1,2$,
satisfy \eqref{Eq:weight1} and \eqref{Eq:weight2}.
Let $\Phi _1,\Phi _2$ be quasi Young functions of order $r_0 \in(0,1]$.
Then the map $(a_1 ,a_2)\mapsto a_1\wpr _Aa_2$ from
$\Sigma _1(\rr {2d})\times \Sigma _1(\rr {2d})$ to $\Sigma _1(\rr {2d})$
is uniquely extendable to a continuous map from
$M^{\Phi _1,\Phi _2}_{(\omega _1)}(\rr {2d}) \times
M^{\infty ,r_0}_{(\omega _2)}(\rr {2d})$ to
$M^{\Phi _1 ,\Phi _2}_{(\omega )}(\rr {2d})$, and
\begin{equation}\label{Eq:main}
\nm {a_1 \wpr _A a_2}{M^{\Phi _1,\Phi _2}_{(\omega )}}
\lesssim
\nm {a_1}{M^{\Phi _1,\Phi _2}_{(\omega _1)}}
\nm {a_2}{M^{\infty ,r_0}_{(\omega _2)}}.
\end{equation}
\end{thm}

\par

We need some preparations for the proof. By \cite[Proposition 3.2]{CheSiTo}
it follows that the map $(A_1,A_2)\mapsto A_1\circ A_2$ is uniquely defined
and continuous from $\mathbb{U}^{\infty ,\infty}_{(\omega _1)}
(\Lambda \times \Lambda )\times
\mathbb{U}^{\infty ,r_0}_{(\omega _2)}(\Lambda \times \Lambda )$
to $\mathbb{U}^{\infty ,\infty}_{(\omega )}(\Lambda \times \Lambda )$
when $\Lambda \subseteq \rr d$ is a lattice, $r_0 \in(0,1]$ 
and $\omega, \omega _1, \omega _2 \in \mascP _E (\rdd )$ satisfy
\begin{equation}\label{Eq:MatrixWeighsComp}
\omega(x,z) \le \omega _1(x,y) \omega _2(y,z), \quad x,y,z\in \rd.
\end{equation}
The following lemma extends certain parts of this continuity to
matrix classes satisfying Orclicz estimates.

\par

\begin{lemma}\label{Lemma1}
Let $\Lambda \subseteq \rr d$ be a lattice,
$\Phi _1,\Phi _2$ be quasi Young functions of order $r_0 \in(0,1]$,
and let $\omega, \omega _1, \omega _2 \in \mascP _E (\rdd)$ satisfy
\eqref{Eq:MatrixWeighsComp}.
Then $(A_1,A_2)\mapsto A_1\circ A_2$
from $\mathbb{U}^{\infty ,\infty}_{(\omega _1)}(\Lambda \times \Lambda )\times
\mathbb{U}^{\infty ,r_0}_{(\omega _2)}(\Lambda \times \Lambda )$
to $\mathbb{U}^{\infty ,\infty}_{(\omega )}(\Lambda \times \Lambda )$
restricts to a continuous map from
$\mathbb{U}^{\Phi _1,\Phi _2}_{(\omega _1)}(\Lambda \times \Lambda )\times
\mathbb{U}^{\infty ,r_0}_{(\omega _2)}(\Lambda \times \Lambda )$
to $\mathbb{U}^{\Phi _1,\Phi _2}_{(\omega )}(\Lambda \times \Lambda )$ and
\begin{equation}\label{Eq:matrix}
  \Vert A_1\circ A_2\Vert _{\mathbb{U}^{\Phi _1,\Phi _2}_{(\omega )}}
  \lesssim
  \Vert A_1\Vert _{\mathbb{U}^{\Phi _1,\Phi _2}_{(\omega _1)}}
   \Vert A_2\Vert _{\mathbb{U}^{\infty ,r_0}_{(\omega _2)}}.
\end{equation}
\end{lemma}

\par

\begin{proof}
Let $A_1=(\mabfa _1(j,k))_{j,k\in \Lambda}$, $A_2=(\mabfa _2(j,k))_{j,k\in \Lambda}$
be matrices, let the matrix elements of $B=A_1\circ A_2$ be denoted by
$\mabfb(j,k)$, and set
\begin{alignat*}{1}
a_m(j,k) &\equiv |\mabfa _m(j,j-k)|\omega _m(j,j-k),\qquad m=1,2,
\intertext{and}
b (j,k) &\equiv |\mabfb(j,j-k)|\omega(j,j-k).
\end{alignat*}
Then
\begin{alignat}{1}
\Vert A_1\Vert _{\mathbb{U}^{\Phi _1,\Phi _2}_{(\omega _1)}}
&=
\Vert a_1\Vert _{\ell^{\Phi _1, \Phi _2}},
\quad
\Vert A_2\Vert _{\mathbb{U}^{\infty ,r_0}_{(\omega _2)}}
=\Vert a_2\Vert _{\ell^{\infty ,r_0}}
\notag
\\[1ex]
\Vert B\Vert _{\mathbb{U}^{\Phi _1,\Phi _2}_{(\omega )}}
&=
\Vert b \Vert _{\ell^{\Phi _1, \Phi _2}}
\notag
\intertext{and}
b (j,k)
&\le
\sum \limits _{m \in \zd} a_1(j,m )a_2(j-m ,k-m ).
\label{Eq:Cond.}
\end{alignat}
By a similar application of Beppo-Levis' theorem, and splitting
up $A_j$ as in Remark \ref{Rem:Aop}, the result follows if we prove
$$
\Vert b  \Vert _{\ell^{\Phi _1, \Phi _2}}
\le
\Vert a_1\Vert _{\ell^{\Phi _1, \Phi _2}}
\Vert a_2\Vert _{\ell^{\infty ,r_0}}.
$$
when $a_1,a_2\in \mathbb
U_0(\Lambda \times \Lambda)$ have
non-negative entries.

\par

Let $\Phi _{0,j}$ be Young functions such that
$\Phi _j(t)=\Phi _{0,j}(t^{r_0})$,
$t\ge 0$, $j=1,2$, and let
$$
c_1(m )=\Vert a_1(\cdo,m )\Vert _{\ell^{\Phi _1}}
\quad \text{and}\quad
c_2(k)=\sup _{j\in \Lambda}a_2(j,k )^{r_0}.
$$
By \eqref{Eq:Cond.} and the fact that $\Phi _{0,1}$ is convex
we get
\begin{multline*}
  \sum \limits _{j\in \Lambda} \Phi _{0,1}
  \left (
  \frac{|b (j,k)|^{r_0}}{\lambda ^{r_0}}
  \right )
  \le
  \sum \limits _{j\in \Lambda} \Phi _{0,1}
  \left (
  \frac 1{\lambda ^{r_0}}
  \sum \limits _{m \in \Lambda}
  a_1(j,m )^{r_0} a_2(j-m ,k-m )^{r_0}
  \right )
  \\[1ex]
  \le \sum \limits _{j\in \Lambda} \Phi _{0,1}
  \left (
  \frac 1{\lambda ^{r_0}}
  \sum \limits _{m \in \Lambda}
  a_1(j,m )^{r_0} c_2(k-m )^{r_0}
  \right )
  \\[1ex]
  =\sum \limits _{j\in \Lambda} \Phi _{0,1}
  \left (
  \sum \limits _{m \in \Lambda}
  \frac{a_1(j,k-m )^{r_0}}{\lambda ^{r_0}} c_2(m )^{r_0}
  \right )
  \\[1ex]
  \le \sum \limits _{j\in \Lambda}
  \left (
  \sum \limits_{m \in \Lambda} \Phi _{0,1}
  \left (
  \frac{a_1(j,k-m )^{r_0}}{\lambda ^{r_0}}
   \right )
  c_2(m )^{r_0}
  \right )
  \\[1ex]
  = \sum \limits _{m \in \Lambda} c_2(m )^{r_0}
  \sum \limits _{j\in \Lambda} \Phi _{0,1}
  \left (
  \frac{a_1(j,k-m )^{r_0}}{\lambda ^{r_0}}
  \right ).
\end{multline*}
This gives
\begin{equation}\label{Eq:Cond.2}
 \Vert b (\cdo ,k) \Vert ^{r_0}_{\ell ^{\Phi _1}}
 \le
 \sum \limits _{m \in \Lambda} c_2(m )^{r_0}
 \Vert a_1(\cdo ,k-m )\Vert ^{r_0}_{\ell ^{\Phi _1}}
 =(c_1^{r_0}*c_2^{r_0})(k),
\end{equation}
in view of the definition of $\ell^{\Phi _1,\Phi _2}$ norm.

\par

By \eqref{Eq:Cond.2} we get
\begin{multline*}
\Vert B\Vert_{\mathbb{U}^{\Phi _1,\Phi _2}_{(\omega )}}
=
\nm b {\ell^{\Phi _1,\Phi _2}}
\le
\left (
\big \Vert c_1^{r_0}*c_2^{r_0} \big \Vert_{\ell ^{\Phi _{0,2}}}
\right)^{\frac 1{r_0}}
\\[1ex]
\le
\left (
\Vert c_1^{r_0}\Vert _{\ell ^{\Phi _{0,2}}}\Vert c_2^{r_0}\Vert _{\ell ^1}
\right )^{\frac 1{r_0}}
=
\Vert a_1\Vert _{\ell ^{\Phi _1,\Phi _2}}
\Vert a_2\Vert _{\ell ^{\infty , r_0}}
\\[1ex]
=
\Vert A_1\Vert_{\mathbb{U}^{\Phi _1,\Phi _2}_{(\omega _1)}}
\Vert A_2\Vert_{\mathbb{U}^{\infty ,r_0}_{(\omega _2)}}. \qedhere
\end{multline*}
\end{proof}

\par

\begin{proof}[Proof of Theorem \ref{Thm:main}]
By Proposition \ref{Prop:CalculiTransferOrlModSp}
we may assume that $A=0$.

\par

Let $\ep >0$, $\phi _1$, $\phi _2$ and $\Lambda$ be the same as in
the proofs of Theorem \ref{Thm:PseudoCont} and Lemma
\ref{Lemma:PseudoCont},
$a_1\in M^{\Phi _1,\Phi _2}_{(\omega _1)}(\rdd )$ and
$a_2\in M^{\infty ,r_0}_{(\omega _2)}(\rdd )$.
By Theorem 2.17 we have
\begin{equation}\label{Eq:Cond.3}
\Vert a_1 \Vert _{M^{\Phi _1,\Phi _2}_{(\omega _1)}}
\asymp
\Vert A_1\Vert _{\mathbb{U}^{\Phi _1,\Phi _2}_{(\vartheta _1)}},
\quad
\Vert a_2 \Vert _{M^{\infty ,r_0}_{(\omega _2)}}
\asymp
\Vert A_2\Vert _{\mathbb{U}^{\infty ,r_0}_{(\vartheta _2)}}
\end{equation}
$$
\op(a_1)=D_{\phi _1}^{\ep ,\Lambda}\circ A_1\circ C_{\phi _2}^{\ep ,\Lambda}
\quad \text{and}\quad
\op(a_2)=D_{\phi _1}^{\ep ,\Lambda}
\circ A_2\circ C_{\phi _2}^{\ep ,\Lambda},
$$
where
$$
A_m=(\mabfa_m(\mabfj ,\mabfk ))
_{\pmb{{j},\mabfk }\in \ep \Lambda ^2},
$$
$$
\mabfa_m(\mabfj ,\mabfk )\equiv e^{i\scal{k-j}{\kappa}}
V_\varphi a_m(j,\kappa,\iota-\kappa,k-j),
\qquad 
\mabfj  = (j,\iota )\in \ep \Lambda ^2,\ \mabfk 
=(k,\kappa ) \in \ep \Lambda ^2
$$
and
$$
\vartheta_m(x,\xi,y,\eta)=\omega_m(x,\eta,\xi-\eta,y-x).
$$
The condition \eqref{Eq:weight1} means for the weights $\vartheta _m$,
$m=0,1,2$,
\begin{equation}\label{Eq:weight3}
\vartheta _0(X,Y)\lesssim \vartheta_1(X,Z) \vartheta_2(Z,Y),\quad X, Y, Z \in \rdd .
\end{equation}
Pick $v_1 \in \mascP_E(\rd)$ such that $\omega_2$ is $v_2$-moderate,
where
$$
v_2 = v_1\otimes v_1\otimes v_1 \otimes v_1 \in \mascP_E(\rr {4d}).
$$
Also let $v =v^2_{1}\otimes v^2_{1} \in \mascP_E(\rdd)$ and
$$
v_0(X,Y)=v(X-Y) \in \mascP_E(\rr {4d}),\quad X,Y \in \rdd.
$$
Then
\begin{equation}\label{Eq:weight4}
\vartheta_2(X,Y)\lesssim
v_0(X,Z)\vartheta_2(Z,Y),\quad X, Y, Z \in\rdd.
\end{equation}

\par

By \eqref{tWigpseudolink} and \eqref{wprtdef} we get
$$
\op(a_1)\circ\op(a_2)=D_{\phi _1}^{\ep ,\Lambda}
\circ A\circ C_{\phi _2}^{\ep ,\Lambda},
$$
where
$$
A=A_1\circ C\circ A_2
$$
and $C=C_{\phi _2}^{\ep ,\Lambda}\circ 
D_{\phi _1}^{\ep ,\Lambda}$
is a matrix of the form
$(\pmb{c}(\mabfj ,\mabfk ))_{\mabfj ,\mabfk \in \ep \Lambda ^2}$
with matrix elements $\pmb{c}(\mabfj ,\mabfk )$,
$\mabfj ,\mabfk \in \ep \Lambda ^2$.

\par

By \cite[Lemma 3.3]{CheSiTo} we get
\begin{multline*}
  \Vert C\Vert_{\mathbb{U}^{\infty ,r_0}_{(v_0)}}
  =\left(
  \sum\limits_{\mabfk \in \ep \Lambda ^2}
  \left(
  \sup\limits_{\mabfj \in \ep \Lambda ^2}
  |\pmb{c}(\mabfj ,\mabfj -\mabfk )v(\mabfk )|^{r_0}
  \right)
  \right)^{\frac 1{r_0}}
  \\[1ex]
  =\left (
  \sum \limits _{\mabfk \in \ep \Lambda ^2}
  |V_{\phi _2}\phi _1(\mabfk )v(\mabfk )|^{r_0}
  \right)^{\frac 1{r_0}}
  \asymp \Vert \phi _1\Vert_{M^{r_0}_{(v)}}< \infty .
\end{multline*}
Thus
$$
C\in \bigcap_{r_0>0}\mathbb{U}^{\infty ,r_0}_{(v_0)}
(\ep \Lambda ^2\times \ep \Lambda ^2).
$$
Then we obtain from
Lemmas \ref{Lemma:PseudoCont} and \ref{Lemma1}
\begin{multline*}
\Vert a_1 \wpr_0
 a_2 \Vert_{M^{\Phi _1,\Phi _2}_{(\omega_0)}}
\asymp \Vert A_1 \circ C \circ A_2\Vert_{\mathbb{U}^{\Phi _1,\Phi _2}_{(\vartheta_0)}}
\\[1ex]
\le \Vert A_1 \circ C\Vert_{\mathbb{U}^{\Phi _1,\Phi _2}_{(\vartheta_1)}}
\Vert A_2\Vert_{\mathbb{U}^{\infty ,r_0}_{(\vartheta_2)}}
\le \Vert A_1\Vert_{\mathbb{U}^{\Phi _1,\Phi _2}_{(\vartheta_1)}}
\Vert C\Vert_{\mathbb{U}^{\infty ,r_0}_{(v_0)}}
\Vert A_2\Vert_{\mathbb{U}^{\infty ,r_0}_{(\vartheta_2)}}
\\[1ex]
\lesssim \Vert A_1\Vert_{\mathbb{U}^{\Phi _1,\Phi _2}_{(\vartheta_1)}}
\Vert A_2\Vert_{\mathbb{U}^{\infty ,r_0}_{(\vartheta_2)}}
\asymp \Vert a_1\Vert_{M^{\Phi _1,\Phi _2}_{(\omega_1)}}
\Vert a_2\Vert_{M^{\infty ,r_0}_{(\omega_2)}}.\qedhere
\end{multline*}
\end{proof}

\par

\end{document}